\documentclass[11pt]{article}
\usepackage{graphicx}
\usepackage{amsmath,amsthm,amssymb,enumerate}
\usepackage[left=2.5cm,right=2.5cm,top=3cm,bottom=3cm]{geometry}
\usepackage{color}
\catcode`\@=11 \@addtoreset{equation}{section}

\catcode`\@=12

\usepackage{euscript,mathrsfs}
\newtheorem{Theorem}{Theorem}[section]
\newtheorem{Proposition}[Theorem]{Proposition}
\newtheorem{Lemma}[Theorem]{Lemma}
\newtheorem{Corollary}[Theorem]{Corollary}
\newtheorem{Remark}[Theorem]{Remark}
\newtheorem{Definition}[Theorem]{Definition}

\newcommand{\bTheorem}[1]{
\begin{Theorem} \label{T#1} }
\newcommand{\eT}{\end{Theorem}}

\newcommand{\bProposition}[1]{
\begin{Proposition} \label{P#1}}
\newcommand{\eP}{\end{Proposition}}

\newcommand{\bLemma}[1]{
\begin{Lemma} \label{L#1} }
\newcommand{\eL}{\end{Lemma}}

\newcommand{\bCorollary}[1]{
\begin{Corollary} \label{C#1} }
\newcommand{\eC}{\end{Corollary}}

\newcommand{\bRemark}[1]{
\begin{Remark} \label{R#1} }
\newcommand{\eR}{\end{Remark}}

\newcommand{\bDefinition}[1]{
\begin{Definition} \label{D#1} }
\newcommand{\eD}{\end{Definition}}

\newcommand{\bFormula}[1]{
\begin{equation} \label{#1}}
\newcommand{\eF}{\end{equation}}

\newcommand{\Ov}[1]{\overline{#1}}

\newcommand{\vr}{\varrho}
\newcommand{\vre}{\vr_\ep}

\newcommand{\vue}{\vu_\ep}

\newcommand{\vu}{\vc{u}}
\newcommand{\vc}[1]{{\bf #1}}
\newcommand{\bfU}{\tilde \vu}

\newcommand{\Div}{{\rm div}_x}
\newcommand{\Grad}{\nabla_x}

\newcommand{\tn}[1]{\mathbb{ #1 }}
\newcommand{\dx}{{\rm d} {x}}
\newcommand{\dt}{{\rm d} t }

\newcommand{\intTor}[1]{\int_{\tor} #1 \ \dx}

\newcommand{\ep}{\varepsilon}
\newcommand{\tor}{\mathcal{T}^3}

\font\F=msbm10 scaled 1000
\newcommand{\R}{\mbox{\F R}}

\definecolor{Cgrey}{rgb}{0.85,0.85,0.85}
\definecolor{Cblue}{rgb}{0.50,0.85,0.85}
\definecolor{Cred}{rgb}{1,0,0}
\definecolor{fancy}{rgb}{0.10,0.85,0.10}

\newcommand\Cbox[2]{%
    \newbox\contentbox%
    \newbox\bkgdbox%
    \setbox\contentbox\hbox to \hsize{%
        \vtop{
            \kern\columnsep
            \hbox to \hsize{%
                \kern\columnsep%
                \advance\hsize by -2\columnsep%
                \setlength{\textwidth}{\hsize}%
                \vbox{
                    \parskip=\baselineskip
                    \parindent=0bp
                    #2
                }%
                \kern\columnsep%
            }%
            \kern\columnsep%
        }%
    }%
    \setbox\bkgdbox\vbox{
        \color{#1}
        \hrule width  \wd\contentbox %
               height \ht\contentbox %
               depth  \dp\contentbox
        \color{black}
    }%
    \wd\bkgdbox=0bp%
    \vbox{\hbox to \hsize{\box\bkgdbox\box\contentbox}}%
    \vskip\baselineskip%
}

\date{}
\newcommand{\expe}[1]{ \mathbb{E} \left[ #1 \right] }
\newcommand{\Dif}{{\rm d}}
\newcommand{\DD}{D^d_t}
\newcommand{\DS}{\mathbb{D}^s_t}
\newcommand{\Psp}{(\Omega, \mathfrak{F}, \mathbb{P})}
\newcommand{\StoB}{\left(\Omega, \mathfrak{F},\left\{\mathfrak{F}_t \right\}_{t \geq 0},  \mathbb{P}\right)}

\newcommand{\N}{\mathbb N}

\newcommand{\E}{\mathbb E}
\newcommand{\p}{\mathbb P}

\newcommand{\dd}{\mathrm{d}}

\newcommand{\bfu}{\mathbf u}
\newcommand{\bfv}{\mathbf v}
\newcommand{\bfq}{\mathbf q}
\newcommand{\bq}{\mathbf q}
\newcommand{\bfQ}{\mathbf Q}

\newcommand{\bfw}{\mathbf w}

\newcommand{\dif}{\mathrm{d}}

\newcommand{\mf}{\mathfrak{F}}
\newcommand{\mr}{{R}}
\newcommand{\prst}{\mathbb{P}}

\newcommand{\mt}{\mathcal{T}^3}

\begin{document}


\title{Compressible fluids driven by stochastic forcing: The relative energy inequality and applications}

\author{Dominic Breit \and Eduard Feireisl \and \and Martina Hofmanov\' a}

\maketitle

\centerline{Department of Mathematics, Heriot-Watt University}

\centerline{Riccarton Edinburgh EH14 4AS, UK}

\bigskip

\centerline{Institute of Mathematics of the Academy of Sciences of the Czech Republic}

\centerline{\v Zitn\' a 25, CZ-115 67 Praha 1, Czech Republic}

\bigskip

\centerline{Technical University Berlin, Institute of Mathematics}

\centerline{Stra\ss e des 17. Juni 136, 10623 Berlin, Germany}

\bigskip






\maketitle

\bigskip





\begin{abstract}

We show the relative energy inequality for the compressible Navier-Stokes system driven by a stochastic forcing. As a corollary, we prove the weak-strong
uniqueness property (pathwise and in law) and convergence of weak solutions in the inviscid-incompressible limit. In particular, we establish a Yamada--Watanabe type result in the context of the compressible Navier-Stokes system, that is, pathwise weak--strong uniqueness implies weak--strong uniqueness in law.

\end{abstract}



\section{Introduction}
\label{i}

The concept of \emph{weak solution} was introduced in
the mathematical fluid mechanics to handle the unsurmountable difficulties related to the hypothetical or effective possibility of
singularities experienced by solutions of the corresponding systems of partial differential equations. However, as
shown in the seminal work of DeLellis and Sz\' ekelyhidi \cite{DelSze3}, the sofar well accepted criteria derived from the
underlying physical principles as the Second law of thermodynamics are not sufficient to guarantee the expected well-posedness of the
associated initial and/or boundary value problems in the class of weak solutions.
The approach based on \emph{relative entropy/energy} introduced by Dafermos \cite{Daf4} has become an important
and rather versatile tool whenever a weak solution is expected to be, or at least to approach, a smooth one,
see Leger, Vasseur \cite{LegVas}, Masmoudi \cite{MaS5}, Saint-Raymond \cite{SaRay} for various applications. In particular,
the problem of weak-strong uniqueness for the compressible Navier-Stokes and the Navier-Stokes-Fourier system
were addressed by Germain \cite{Ger} and finally solved in \cite{FeiNov10}, \cite{FENOSU}.

All the aforementioned results apply to the {deterministic models}. Our goal is to adapt the
concept of relative energy/entropy \color{black} to the stochastic setting. As a model example,
we consider the Navier-Stokes system describing the motion of a compressible viscous fluid driven by stochastic forcing:

\Cbox{Cgrey}{
\begin{eqnarray}
  \Dif \vr + \Div (\vr \vu) \ \dt &=& 0, \label{eq1} \\
  \label{eq2} \Dif (\vr \vu) + \left[ \Div (\vr \vu \otimes \vu) + \Grad p(\vr) \right] \Dif t &=& \Div \mathbb{S}(\Grad \vu) \ \dt + \mathbb{G}(\vr, \vr \vu) \Dif W, \\
  \tn{S} (\Grad \vu) &=& \mu \Big(\Grad \vu + \Grad^t \vu -  \frac{2}{3} \Div \vu \mathbb{I} \Big) + \eta \Div \vu \tn{I},\label{eq3}
\end{eqnarray}
}

\noindent where $p = p(\vr)$ is the pressure, $\mu > 0$, $\eta \geq 0$ the viscosity coefficients, and the
driving force is represented
by a cylindrical Wiener process $W$ in a separable Hilbert space $\mathfrak{U}$ defined on some probability space $\Psp$. We assume that $W$ is formally given by the expansion
$$W(t)=\sum_{k\geq 1} e_k W_k(t),$$
where $\{ W_k \}_{k \geq 1}$ is a family a family of mutually independent real-valued Brownian motions and $\{e_k\}_{k\geq 1}$ is an orthonormal basis of $\mathfrak{U}$. We assume that $\mathbb{G}(\vr, \vr\vu) $ belongs to the class of Hilbert-Schmidt operators $L_2(\mathfrak U;L^2(\mathcal T^3))$ a.e. in $(\omega,t)$. The precise description will be given in Section \ref{M}. The stochastic forcing then takes the form
\[
\mathbb{G}(\vr, \vr\vu) \Dif W = \sum_{k \geq 1} \vc{G}_k (\vr, \vr\vu) \,\dd W_k.
\]
Our main goal is to derive a relative energy inequality for system (\ref{eq1}--\ref{eq3}) analogous to that obtained in the deterministic case in
\cite{FENOSU}. For the sake of simplicity, we focus on the space-periodic boundary conditions yielding the physical space
in the form of the ``flat'' torus
\[
\mathcal{T}^N = \Big([-1,1]\Big|_{\{-1,1\}} \Big)^N.
\]
Moreover, we restrict ourselves to the physically relevant case $N=3$ seeing that our arguments can be easily adapted for $N=1,2$.

We proceed in several steps:

\begin{itemize}

\item
Revisiting the existence proof in \cite{BrHo} we derive a weak differential form of
the \emph{energy inequality} associated to system (\ref{eq1}--\ref{eq3}):
\begin{equation} \label{EI2}
\begin{split}
&- \int_0^T \partial_t \psi
\bigg( \intTor{ \Big[ \frac{1}{2} \varrho | {\bf u} |^2 + H(\varrho) \Big] } \bigg) \ \dt
+ \int_0^T \psi \intTor{  \mathbb{S} (\nabla {\bf u}): \nabla {\bf u} } \ {\rm d}t \\
\leq &\psi(0) \intTor{ \Big[ \frac{| (\vr \vu)(0, \cdot) |^2 }{2 \varrho(0, \cdot)}  + H(\varrho(0, \cdot)) \Big] }
+ \frac{1}{2} \int_0^T
\psi \bigg(
\intTor{ \sum_{k \geq 1} \frac{ | {\bf G}_k (\varrho, \varrho {\bf u}) |^2 }{\varrho} } \bigg) {\rm d}t
\\ &+ \int_0^T  \psi {\rm d}M_E
\end{split}
\end{equation}
holds true $\mathbb{P}$-a.s. for any deterministic smooth test function $\psi \geq 0$, $\psi(T) = 0$, where
\[
H(\vr) = \vr \int_0^\varrho \frac{p(z)}{z^2} \ {\rm d}z
\]
is the pressure potential, and
$M_E$ is a real-valued martingale satisfying
\[
\E\bigg[ \sup_{t \in [0,T]} |M_E|^p \bigg] \leq c(p) \bigg( 1 + \expe{\intTor{ \left( \frac{|(\vr \vu)(0, \cdot)|^2 }{2 \vr(0, \cdot)} +  H(\vr(0, \cdot)) \bigg) } }^p
\right)
\]
for any $1 \leq p < \infty$,
see Section \ref{REI}.
\item
We introduce \color{black}
the \emph{relative energy} functional
\begin{equation} \label{eq:entrdef}
\mathcal{E} \left( \vr, \vu \Big| r , {\bf U} \right) = \intTor{ \Big[ \frac{1}{2} \varrho | \vu - {\bf U} |^2 + H(\varrho) -
H'(r) (\varrho - r) - H(r) \Big] },
\end{equation}
that may be viewed as a kind of distance between
a weak martingale solution $[\vr, \vu]$ of system (\ref{eq1}--\ref{eq3}) and a pair of arbitrary (smooth)
processes \color{black} $[r, {\bf U}]$.
In view of future applications, it is convenient that the behavior of the test functions $[r, {\bf U}]$ mimicks that of $[\vr, \vu]$. Accordingly,
we require $r$ and $\vc{U}$ to be stochastic processes adapted to $\{\mathfrak{F}_t\}$:
\begin{equation} \label{difer}
\Dif r = \DD r \,\dt + \DS r \,\Dif W, \qquad \Dif \vc{U} = \DD \vc{U} \,\dt + \DS \vc{U} \,\Dif W.
\end{equation}
We assume that $\DD r,\DD \vc{U}$ are functions of $(\omega,t,x)$ and that $\DS r,\DS\vc{U}$ belong to $L_2(\mathfrak U;L^2(\mathcal T^3))$ a.e. in $(\omega,t)$. Both with appropriate integrability and pace-regularity.
Under these circumstances, the \emph{relative energy inequality} reads:

\Cbox{Cgrey}{

\begin{eqnarray}
\label{REI2}
- \int_0^T \partial_t \psi \ \mathcal{E} \left( \vr , \vu \Big| r , \vc{U} \right)\ \dt   &+&
\int_0^T \psi \intTor{  \left( \mathbb{S} (\nabla {\bf u}) - \tn{S}(\Grad \vc{U} \right): (\Grad {\bf u} -
\Grad \vc{U} )} \ {\rm d}t \\
\nonumber
&\leq&  \psi(0) \mathcal{E} \left( \vr, \vu \ \Big| r, \vc{U} \right) (0) + \int_0^T  \psi {\rm d}M_{RE}
+ \int_0^T
\psi \mathcal{R}\left( \vr , \vu \Big| r , \vc{U} \right)  {\rm d}t,
\end{eqnarray}
for any $\psi$ belonging to the same class as in (\ref{EI2}). Here, similarly to (\ref{EI2}), $M_{RE}$ is a real-valued square integrable martingale.

}

\noindent
The remained term is
\begin{align}
\mathcal{R}  \left( \varrho, \vu \Big| r , {\bf U} \right) &=\intTor{ \tn{S}(\Grad \vc{U}):(\Grad \vc{U}-\Grad \vu)} +\intTor{ \varrho\Big(\DD \vc{U}+\vu \cdot\Grad \vc{U} \Big)(\vc{U} -\vu) }\nonumber\\
&+\intTor{ \big((r-\varrho)H''(r) \DD r +\Grad H'(r)(r \vc{U} -\varrho\vu)\big)} -\intTor{ \Div \vc{U} (p(\varrho)-p(r))} \nonumber\\
&+\frac{1}{2}\sum_{k\geq1}\,\intTor{ \varrho \Big|\frac{\vc{G}_k(\varrho,\varrho \vu)}{\varrho}  - \DS {\bf U}(e_k) \Big|^2 }\nonumber\\
&+ \frac{1}{2} \sum_{k\geq1} \intTor{ \vr H'''(r) |  \DS r(e_k)  |^2 }  + \frac{1}{2}
\sum_{k\geq1}\intTor{ p''(r) |  \DS r (e_k) |^2 } .\label{rem}
\end{align}

The relative energy inequality is proved in Section \ref{REI}. The main ingredients of the proof are the energy inequality (\ref{EI2}) and a careful
application of It\^{o}'s stochastic calculus.

\item As a corollary of the relative energy inequality we present two applications:
The weak-strong uniqueness property (pathwise and in law) for the stochastic Navier-Stokes system (\ref{eq1}--\ref{eq3}) in Section \ref{WS}, and
the singular incompressible-inviscid limit in Section \ref{II}. In particular, we establish a Yamada--Watanabe type result that says, roughly speaking, that pathwise weak--strong uniqueness implies weak--strong uniqueness in law, see Theorem \ref{thm:uniqlaw}.

\end{itemize}

\begin{Remark} \label{first}

A weak martingale solution satisfying
the energy inequality in the ``differential form'' \eqref{EI2} may be seen as an analogue of the \emph{a.s. super--martingale}
solution introduced by Flandoli and Romito \cite{FlaRom} and further developed by Debussche and Romito \cite{DebRom} in the context of the
incompressible Navier-Stokes system.

It follows from \eqref{EI2} that the limits
\[
{\rm ess}\lim_{\tau \to s+} \intTor{ \Big[ \frac{1}{2} \vr |\vu|^2 + H(\vr) \Big] (\tau) }, \
{\rm ess}\lim_{\tau \to t-} \intTor{ \Big[ \frac{1}{2} \vr |\vu|^2 + H(\vr) \Big] (\tau) }
\]
exist $\mathbb{P}-$a.s. for a.a. $0 \leq s \leq t \leq T$ including $s=0$,
\[
\lim_{\tau \to 0+} \intTor{ \Big[ \frac{1}{2} \vr |\vu|^2 + H(\vr) \Big] (\tau) } =
\intTor{ \Big[ \frac{1}{2} \vr |\vu|^2 + H(\vr) \Big] (0) },
\]
and
\begin{equation} \label{EI2+}
\begin{split}
\left[ \intTor{ \left[ \frac{1}{2} \vr |\vu|^2 + H(\vr) \right] (\tau) } \right]_{\tau = s}^{\tau = t}
&+ \int_{s}^t \intTor{ \mathbb{S}(\Grad \vu) : \Grad \vu } \ \dt
\\
\leq \frac{1}{2} \int_s^t \intTor{ \sum_{k \geq 1} \frac{ \left| \vc{G}_k (\vr, \vr \vu) \right|^2 }{\vr} } \ \dt
&+ M_E(t) - M_E(s)\ \mbox{$\mathbb{P}$-a.s.}
\end{split}
\end{equation}
Finally, in view of the weak lower-semicontinuity of convex functionals,
\[
\liminf_{\tau \to t-} \intTor{ \left[ \frac{1}{2} \vr |\vu|^2 + H(\vr) \right] (\tau) }
\geq \intTor{ \left[ \frac{1}{2} \vr |\vu|^2 + H(\vr) \right] (t) } \ \mbox{for any}\ t \in [0,T)\ \mbox{$\mathbb{P}$-a.s.}
\]
Similar observations hold for the relative energy inequality (\ref{REI2}) that can be rewritten as
\begin{equation}
\label{REI2+}
\begin{split}
\mathcal{E} \left( \vr , \vu \Big| r , \vc{U} \right)(t)   &+
\int_s^t \intTor{  \left( \mathbb{S} (\nabla {\bf u}) - \tn{S}(\Grad \vc{U} \right): (\Grad {\bf u} -
\Grad \vc{U} )} \ {\rm d}r \\
&\leq  \mathcal{E} \left( \vr , \vu \Big| r , \vc{U} \right)(s) + M_{RE}(t) - M_{RE}(s)
+ \int_s^t
\mathcal{R}\left( \vr , \vu \Big| r , \vc{U} \right)  {\rm d}r,
\end{split}
\end{equation}
for any $0 \leq t \leq T$, a.a. $0 \leq s \leq t$ including $s=0$ $\mathbb{P}$-a.s.

\end{Remark}

\color{black}

\section{Mathematical framework and main results}
\label{M}

Throughout the whole text, we suppose that the pressure $p = p(\vr)$ belongs to the class $p \in C^1[0, \infty) \cap C^3(0, \infty)$ and
satisfies
\begin{equation}
\label{press}
p(0) = 0, \ p'(\varrho) > 0 \ \mbox{if}\ \vr > 0, \ \lim_{\varrho \to \infty} \frac{p'(\vr)}{\vr^{\gamma - 1}} = p_\infty > 0,\
\gamma > \frac{3}{2}.
\end{equation}

Next we specify the stochastic forcing term.
Let $(\Omega,\mf,(\mf_t)_{t\geq0},\prst)$ be a stochastic basis with a complete, right-continuous filtration. The process $W$ is a cylindrical Wiener process, that is,
$$W(t)=\sum_{k\geq 1} e_k W_k(t),$$
where $\{ W_k \}_{k \geq 1}$ is a family a family of mutually independent real-valued Brownian motions and $\{e_k\}_{k\geq 1}$ is an orthonormal basis of $\mathfrak{U}$
To give the precise definition of the diffusion coefficient $\mathbb{G}$, consider $\rho\in L^\gamma(\mt)$, $\rho\geq0$, and $\bfv\in L^2(\mt)$ such that $\sqrt\rho\bfv\in L^2(\mt)$. We recall that we assume $\gamma>\frac{3}{2}$.
Denote $\bfq=\rho\bfv$ and let $\,\mathbb{G}(\rho,\bq):\mathfrak{U}\rightarrow L^1(\mt)$ be defined as follows
$$\mathbb{G}(\rho,\bq)e_k=\mathbf{G}_k(\cdot,\rho(\cdot),\bq(\cdot)).$$
The coefficients $\mathbf{G}_{k}:\mt\times\mr\times\mr^3\rightarrow\mr^3$ are $C^1$-functions that satisfy uniformly in $x\in\mt$
\begin{align}
\vc{G}_k (\cdot, 0 , 0) &= 0 \label{FG1}\\
| \partial_\vr \vc{G}_k | + |\nabla_{\vc{q}} \vc{G}_k | &\leq \alpha_k, \quad \sum_{k \geq 1} \alpha_k  < \infty.
\label{FG2}
\end{align}
As in \cite{BrHo}, we understand the stochastic integral as a process in the Hilbert space $W^{-\lambda,2}(\mt)$, $\lambda>3/2$. Indeed, it can be checked that under the above assumptions on $\rho$ and $\bfv$, the mapping $\mathbb{G}(\rho,\rho\bfv)$ belongs to $L_2(\mathfrak{U};W^{-\lambda,2}(\mt))$, the space of Hilbert-Schmidt operators from $\mathfrak{U}$ to $W^{-b,2}(\mt)$.
Consequently, if\footnote{Here $\mathcal{P}$ denotes the predictable $\sigma$-algebra associated to $(\mf_t)$.}
\begin{align*}
\rho&\in L^\gamma(\Omega\times(0,T),\mathcal{P},\dif\prst\otimes\dif t;L^\gamma(\mt)),\\
\sqrt\rho\bfv&\in L^2(\Omega\times(0,T),\mathcal{P},\dif\prst\otimes\dif t;L^2(\mt)),
\end{align*}
and the mean value $(\rho(t))_{\mt}$ is essentially bounded
then the stochastic integral
\[
\int_0^t \mathbb{G}(\vr, \vr \vu) \ {\rm d} W = \sum_{k \geq 1}\int_0^t \vc{G}_k (\cdot, \vr, \vr \vu) \ {\rm d} W_k
\]
 is a well-defined $(\mf_t)$-martingale taking values in $W^{-\lambda,2}(\mt)$. Note that the continuity equation \eqref{eq1} implies that the mean value $(\varrho(t))_{\mt}$ of the density $\varrho$ is constant in time (but in general depends on $\omega$).
Finally, we define the auxiliary space $\mathfrak{U}_0\supset\mathfrak{U}$ via
$$\mathfrak{U}_0=\bigg\{v=\sum_{k\geq1}\alpha_k e_k;\;\sum_{k\geq1}\frac{\alpha_k^2}{k^2}<\infty\bigg\},$$
endowed with the norm
$$\|v\|^2_{\mathfrak{U}_0}=\sum_{k\geq1}\frac{\alpha_k^2}{k^2},\quad v=\sum_{k\geq1}\alpha_k e_k.$$
Note that the embedding $\mathfrak{U}\hookrightarrow\mathfrak{U}_0$ is Hilbert-Schmidt. Moreover, trajectories of $W$ are $\prst$-a.s. in $C([0,T];\mathfrak{U}_0)$.

%

\subsection{Weak martingale solutions}

The existence of (finite energy) \emph{weak martingale solutions} to the stochastic Navier-Stokes system (\ref{eq1}--\ref{eq3})
was recently established in \cite{BrHo}. We point out that the stochastic basis as well as the Wiener process is an integral part of the martingale solution.
In particular, a martingale solution attains the prescribed initial data
only in law, specifically, if $\Lambda$ is a Borel probability measure on the space $L^\gamma (\tor) \times L^{\frac{2 \gamma}{\gamma + 1}}(\tor;R^3)$ then we may require that
\begin{equation} \label{init}
\mathbb{P}\circ \left( \vr(0), \vr \vu (0) \right)^{-1} = \Lambda.
\end{equation}

Denote $\left< \cdot, \cdot \right>$ the standard duality product between $W^{\lambda,2}(\tor)$, $W^{-\lambda,2}(\tor)$ that coincides with
the $L^2$ scalar product for $\lambda = 0$. Let us recall the definition of a weak martingale solution.

\begin{Definition}
\label{DM1} A quantity
\[
\left[ \StoB ; \vr, \vu, W \right]
\]
is called a weak martingale solution to problem (\ref{eq1}--\ref{eq3}) with the initial law $\Lambda$ provided:
\begin{itemize}
\item
$\StoB$ is a stochastic basis with a complete right-continuous filtration;
\item $W$ is an $\{ \mathfrak{F}_t \}_{t \geq 0}$-cylindrical Wiener process;
\item the density $\vr$ satisfies $\vr \geq 0$, $t \mapsto \left< \vr(t, \cdot), \psi \right> \in C[0,T]$ for any
$\psi \in C^\infty(\tor)$
$\mathbb{P}-$a.s., the function $t \mapsto \left< \vr(t, \cdot), \psi \right>$
is progressively measurable,
and
\[
\E\bigg[\sup_{t \in [0,T]} \| \vr(t,\cdot) \|^p_{L^\gamma(\tor)} \bigg] < \infty \ \mbox{for all}\ 1 \leq p < \infty;
\]
\item the velocity field $\vu$ is adapted, $\vu \in L^2(\Omega \times (0,T); W^{1,2}(\tor;R^3))$,
\[
\E\bigg[\bigg( \int_0^T \| \vu \|^2_{W^{1,2}(\tor; R^3)} \ \dt \bigg)^p \bigg] < \infty\ \mbox{for all}\ 1 \leq p < \infty;
\]
\item the momentum $\vr \vu$ satisfies $t \mapsto \left< \vr \vu, \phi \right> \in C[0,T]$ for any $\phi \in C^\infty(\tor;R^3)$
$\mathbb{P}-$a.s., the function $t \mapsto \left< \vr \vu, \phi \right>$ is progressively measurable,
\[
\expe{ \sup_{t \in [0,T]} \left\| \vr \vu \right\|^p_{L^{\frac{2 \gamma}{\gamma + 1}}} } < \infty\ \mbox{for all}\  1 \leq p < \infty;
\]
\item $\Lambda=\mathbb{P}\circ \left( \vr(0), \vr \vu (0) \right)^{-1} $,
\item for all test functions $\psi \in C^\infty(\tor)$, $\phi \in C^\infty(\tor; R^3)$ and all $t \in [0,T]$ it holds $\mathbb{P}$-a.s.:
\begin{eqnarray}
\nonumber
\Dif \left< \vr, \psi \right> &=& \left< \vr \vu , \Grad \psi \right> \ \dt,
\\
\nonumber \Dif \left< \vr \vu, \phi \right>  &=& \Big[ \left< \vr \vu \otimes \vu , \Grad \phi \right>  - \left< \tn{S}(\Grad \vu), \Grad \phi \right>
+ \left< p(\vr), \Div \phi \right> \Big] \dt
+ \left< \mathbb{G} (\vr, \vr \vu) , \phi \right> \Dif W;
\end{eqnarray}

\end{itemize}

\end{Definition}

\color{black}

The following existence result was proved in \cite{BrHo}:

\begin{Theorem} \label{thm:exist}
Let the pressure $p$ be as in (\ref{press}) and let $\vc{G}_k$ be continuously differentiable satisfying (\ref{FG1}), (\ref{FG2}).
Let the initial law $\Lambda$ be given on the space $L^\gamma (\tor) \times L^{\frac{2 \gamma}{\gamma + 1}}(\tor;R^3)$ and
\begin{eqnarray}
\nonumber
\Lambda \Big\{ (\vr, \vc{q} ) \in L^\gamma (\tor) & \times & L^{\frac{2 \gamma}{\gamma + 1}}(\tor;R^3),\ \vr \geq 0,  \\
\nonumber  0 < M_1 \leq \intTor{ \vr } \leq M_2,\ \vc{q} &=& 0 \ \mbox{a.e. on the set} \ \{ \vr = 0\} \Big\} = 1 ,
\end{eqnarray}
for certain constants $0<M_1<M_2$,
\[
\int_{ L^\gamma \times L^{2\gamma/(\gamma + 1)}} \left\| \frac{1}{2} \frac{ |\vc{q}|^2}{\vr}  + H(\vr) \right\|_{L^1(\tor)}^p
\ {\rm d} \Lambda (\vr, \vc{q} ) \leq c(p) < \infty
\]
for any $1 \leq p < \infty$.

Then the Navier-Stokes system (\ref{eq1}--\ref{eq3}) possesses at least one weak martingale solution with the initial law
(\ref{init}).
In addition, the equation of continuity (\ref{eq1}) holds also in the renormalized sense
\[
\Dif \left< b(\vr), \psi \right> = \left< b(\vr) \vu, \Grad \psi \right> \ {\rm d}t - \left<  \left( b(\vr) - b'(\vr) \vr \right) \Div \vu,
\psi \right> \ \dt
\]
for any test function $\psi \in C^\infty(\tor)$, and any $b \in C^1[0,\infty)$, $b'(\vr) = $ for $\vr \geq \vr_g$.
Moreover, the energy estimates
\begin{eqnarray}
\label{energy}
\E\bigg[ \sup_{t \in [0,T]} \bigg( \intTor{ \Big[ \frac{|\vr \vu|^2 }{2 \vr} + H(\vr) \Big] } \bigg)^p \bigg]
  &+& \E\bigg[ \bigg( \int_0^T \intTor{ \tn{S}(\Grad \vu) : \Grad \vu } \ \dt \bigg)^p \bigg] \\
\nonumber &\leq&
c(p) \,\E\bigg[ \bigg( \intTor{ \Big[ \frac{|(\vr \vu) (0, \cdot)|^2 }{2 \vr(0, \cdot)} +  H(\vr(0, \cdot)) \Big] } \bigg)^p + 1 \bigg]
\end{eqnarray}
hold
for any $1 \leq p < \infty$. Because of \eqref{energy} this solution is called finite energy weak martingale solution.
\end{Theorem}

\begin{Remark}
\label{R1}

Note that the energy
\[
\intTor{ \left( \frac{1}{2} {\vr |\vu|^2 } + H(\vr) \right) }
\]
is a priori defined only for a.a. $t \in (0,T)$ while
\[
[\vr, \vc{q}] \mapsto  \frac{| \vc{q}|^2 }{2\vr} + H(\vr)
\]
is a convex function of its arguments and the composition
\[
\intTor{ \left( \frac{|\vr \vu|^2 }{2 \vr} + H(\vr) \right) }
\]
is therefore defined for \emph{any} $t \in [0,T]$ $\mathbb{P}-$a.s. Moreover, we have
\[
\intTor{ \left( \frac{1}{2} {\vr |\vu|^2 } + H(\vr) \right) } = \intTor{ \left( \frac{|\vr \vu|^2 }{2 \vr} + H(\vr) \right) }
\ \mbox{a.e. in}\ (0,T)
\]
and
\[
\expe{ \left( \intTor{ \left( \frac{|(\vr \vu) (0, \cdot)|^2 }{2 \vr(0, \cdot)} +  H(\vr(0, \cdot)) \right) } \right)^p }
= \int_{ L^\gamma \times L^{2\gamma/(\gamma + 1)}} \left\| \frac{1}{2} \frac{ |\vc{q}|^2}{\vr}  + H(\vr) \right\|_{L^1(\tor)}^p
\ {\rm d} \Lambda (\vr, \vc{q} )
\]
for any martingale solution with the initial law $\Lambda$.

\color{black}

\end{Remark}

\subsection{Energy inequality}

The piece of information provided by (\ref{energy}) is not sufficient for proving the relative energy inequality in the form suitable
for applications. Our first goal is therefore to prove a refined version of (\ref{energy}). Revisiting the original existence proof in \cite{BrHo} we deduce the following result proved in Section \ref{EI} below.

\begin{Proposition}
\label{prop:energynew}
Under the hypotheses of Theorem \ref{thm:exist}, let $\left( \StoB, \vr, \vu, W \right)$ be the finite energy weak martingale solution
constructed via the scheme proposed in \cite{BrHo}.
Then there exists a real-valued martingale $M_E$, satisfying
\[
\E\bigg[ \sup_{t \in [0,T]} |M_E|^p \bigg] \leq c(p) \left( 1 + \E\bigg[\intTor{ \left( \frac{|(\vr \vu)(0, \cdot)|^2 }{2 \vr(0, \cdot)} +  H(\vr(0, \cdot)) \right) } \bigg]^p
\right)
\]
for any $1 \leq p < \infty$ such that the energy inequality (\ref{EI2}) holds for any
spatially homogeneous ($x$-independent) deterministic function  $\psi$,
\begin{equation} \label{testf}
\psi \in W^{1,1} [0,T],\
\psi \geq 0,\ \psi (T) = 0, \ \int_0^T | \partial_t \psi | \ \dt  < \infty .
\end{equation}

\end{Proposition}

\color{black}

\begin{Definition} \label{D2}

A weak martingale solution of problem (\ref{eq1}--\ref{eq3}) satisfying the energy inequality (\ref{EI2}) will be called
\emph{dissipative} martingale solution.

\end{Definition}

\subsection{Relative energy/entropy inequality}

Our main result is the following theorem.

\begin{Theorem} \label{thREI}
Under the hypothesis of Theorem \ref{thm:exist}, let
\[
\left[ \StoB; \vr, \vu, W \right]
\]
be a dissipative martingale solution of problem (\ref{eq1}--\ref{eq3}) in $[0,T]$. Suppose that functions $r$, $\vc{U}$ are random processes
adapted to $\{ \mathfrak{F}_t \}_{t \geq 0}$,
\[
r \in C([0,T]; W^{1,q}(\tor)), \ \vc{U} \in C([0,T]; W^{1,q}(\tor, R^3)) \ \quad\text{$\mathbb{P}$-a.s. for all}\quad 1 \leq q < \infty,
\]
\[
\E\bigg[\sup_{t \in [0,T] } \| r \|_{W^{1,q}(\tor)}^2\bigg]^q  + \E\bigg[ \sup_{t \in [0,T] } \| \vc{U} \|_{W^{1,q}(\tor;R^3)}^2\bigg]^q \leq c(q),
\]
\begin{equation} \label{bound}
0 < \underline{r} \leq r(t,x) \leq \overline{r} \quad\text{$\mathbb{P}$-a.s.},
\end{equation}
Moreover, $r$, $\vc{U}$ satisfy (\ref{difer}), where
\begin{align*}
&\DD r, \DD \vc{U}\in L^q(\Omega;L^q(0,T;W^{1,q}(\mt))),\quad
\DS r,\DS \vc{U}\in L^2(\Omega;L^2(0,T;L_2(\mathfrak U;L^2(\tor)))),\\
&\bigg(\sum_{k\geq 1}|\DS r(e_k)|^q\bigg)^\frac{1}{q},\bigg(\sum_{k\geq 1}|\DS \vc{U}(e_k)|^q\bigg)^\frac{1}{q}\in L^q(\Omega;L^q(0,T;L^{q}(\mt))).
\end{align*}

Then the relative energy inequality (\ref{REI2}), (\ref{rem}) holds for any $\psi$ satisfying (\ref{testf}), where
the norm of the martingale $M_R$ depends only on the norms of $r$ and $\vc{U}$ in the aforementioned spaces.

\end{Theorem}

\color{black}

\begin{Remark} \label{Rem2}

Hypothesis (\ref{bound}) seems rather restrictive and even unrealistic in view of the expected properties
of random processes. On the other hand, it is necessary to handle the compositions of the non-linearities, in particular the pressure
$p = p(r)$. Note that (\ref{bound}) can always be achieved replacing $r$ by $\tilde r$, where
\[
\tilde r(t) = r (t \wedge \tau_{\underline{r}, \Ov{r}}),
\]
where $\tau_{\underline{r}, \Ov{r}}$ is a stopping time,
\[
\tau_{\underline{r}, \Ov{r}} = \inf_{t\in[0,T]} \left\{ \inf_{\tor} r (t, \cdot) < \underline{r} \ \mbox{or}\
\sup_{\tor} r (t, \cdot) > \Ov{r} \right\}.
\]

\end{Remark}

\color{black}

\begin{Remark} \label{Rem3}

For the sake of simplicity, we prove Theorem \ref{thREI} in the natural 3D-setting. The same result holds in the dimensions 1 and 2
as well.

\end{Remark}

\color{black}

Theorem \ref{thREI} will be proved in the next section.

\section{Relative energy inequality}
\label{REI}

Our goal in this section is to prove Theorem \ref{thREI}.

\subsection{Energy inequality - proof of Proposition \ref{prop:energynew} }
\label{EI}

The main objective of this section is the proof of the energy inequality (\ref{EI2}) claimed in Proposition \ref{prop:energynew}. To this end, we adapt the construction of the martingale solution in \cite{BrHo}.
First, let us briefly recall the method of the proof of \cite[Theorem 2.2]{BrHo}.
It is based on a four layer approximation scheme: the continuum equation is regularized by means of an artificial viscosity $\ep \Delta \vr$ and the momentum equation is modified correspondingly so that the energy inequality is preserved. In addition, an artificial pressure term $\delta \Grad \vr^\beta$ to (\ref{eq2}) to weaken the hypothesis upon the adiabatic constant $\gamma$.
The aim is to pass to the limit first in $\varepsilon\rightarrow0$ and subsequently in $\delta\rightarrow0$, however, in order to solve the approximate problem for $\varepsilon>0$ and $\delta>0$ fixed two additional approximation layers are needed. In particular, a stopping time technique is employed to establish the existence of a unique solution to a finite-dimensional approximation, the so called Faedo-Galerkin approximation, on each random time interval $[0,\tau_R)$ where the stopping time $\tau_R$ is defined as
$$\tau_R=\inf\big\{t\in[0,T];\|\bfu\|_{L^\infty}\geq R\big\}\wedge\inf\bigg\{t\in[0,T];\bigg\|\int_0^t\mathbb{G}^{N}\big(\varrho,\varrho\bfu\big)\,\dif W\bigg\|_{L^\infty}\geq R\bigg\}$$
(with the convention $\inf\emptyset=T$), where $\mathbb{G}$ is a suitable finite-dimensional approximation of $\mathbb{G}$. It is then showed that the blow up cannot occur in a finite time so letting $R\rightarrow\infty$ gives a unique solution to the Faedo-Galerkin approximation on the whole time interval $[0,T]$. 
The remaining passages to the limit, i.e. $N\rightarrow\infty$, $\varepsilon\to0$ and $\delta\to0$, are justified via the stochastic compactness method.

\medskip

{\it First approximation level:}

\medskip

To simplify notation, we drop the indexes $N$, $\ep$, and $\delta$ and denote
$\vr$, $\vu$ the basic family of approximate solutions constructed in \cite[Subsection 3.1]{BrHo}, specifically,
they solve the fixed point problem \cite[(3.6)]{BrHo} on a corresponding random time interval $[0,\tau_R)$. Inspecting the proof of
\cite[Proposition 3.1]{BrHo} we deduce
\begin{equation} \label{ei1}
\begin{split}
{\rm d} \left( \intTor{ \Big[ \frac{1}{2} \varrho | {\bf u} |^2 + H_\delta(\varrho) \Big] } \right) &+ \left( \intTor{ \mathbb{S} (\nabla {\bf u}): \nabla
{\bf u} } \right) \ {\rm d}t
\\
\leq  &\left( \intTor{ \vu \cdot {\tn{G}}^N (\varrho, \varrho \vu ) } \right) {\rm d} W + \frac{1}{2} \bigg( \sum_{k \geq 1}
\intTor{ \frac{ | {\bf G}_k (\varrho, \varrho {\bf u}) |^2 }{\varrho} } \bigg) {\rm d}t,
\end{split}
\end{equation}
where
\[
H_\delta (\vr) = H(\vr) + \frac{\delta}{\beta - 1} \vr^\beta,
\]
and $\tn{G}^N(\vr, \vr \vu)$ is the approximation of $\tn{G}(\vr, \vr \vu)$ introduced in \cite[formula (3.2)]{BrHo}. It follows from
\cite[Corollary 3.2]{BrHo} that (\ref{ei1}) holds on the whole time interval $[0,T]$.

Now we may apply It\^{o}'s product formula to compute
\[
\Dif \Big[\Big( \frac{1}{2} \varrho | {\bf u} |^2 + H_\delta (\varrho) \Big) \psi \Big],
\]
where $\psi$ is a spatially homogeneous test function satisfying (\ref{testf}):
\[
\begin{split}
{\rm d} \Big( \psi \int_{\tor} \Big[ \frac{1}{2} \varrho | {\bf u} |^2 &+ H_\delta (\varrho) \Big] \ \dx  \Big)
\\
&= \left( \intTor{ \Big[ \frac{1}{2} \varrho | {\bf u} |^2 + H_\delta (\varrho) \Big] } \ \partial_t \psi \right) {\rm d}t +
\psi \,{\rm d} \left( \intTor{ \Big[ \frac{1}{2} \varrho | {\bf u} |^2 + H_\delta(\varrho) \Big] } \right)
\\
&\leq \left( \intTor{ \Big[ \frac{1}{2} \varrho | {\bf u} |^2 + H_\delta(\varrho) \Big] } \ \partial_t \psi \right) {\rm d}t
- \left( \intTor{ \mathbb{S} (\nabla {\bf u}): \nabla
{\bf u} } \ \psi \right)  {\rm d}t
\\
&+  \left( \psi \ \intTor{ {\bf u} \cdot {\tn{G}}^N(\varrho, \varrho {\bf u}) } \right) {\rm d}W + \frac{1}{2} \bigg( \sum_{k \geq 1}
\psi \ \intTor{ \frac{ | {\bf G}_k (\varrho, \varrho {\bf u}) |^2 }{\varrho} } \bigg) {\rm d}t.
\end{split}
\]
Thus we may integrate with respect to time to obtain
\begin{equation} \label{ei2}
\begin{split}
\int_0^T \psi \intTor{  \mathbb{S} (\nabla {\bf u}): \nabla {\bf u} } \ {\rm d}t &\leq \psi(0) \intTor{ \Big[
\frac{ |(\vr \vu)(0, \cdot) |^2 }{2\vr(0, \cdot)} + H_\delta(\varrho(0, \cdot)) \Big] }
\\
&+ \int_0^T \partial_t \psi
\left( \intTor{ \Big[ \frac{1}{2} \varrho | {\bf u} |^2 + H_\delta(\varrho) \Big] } \ \right) {\rm d}t
\\
+  \int_0^T  \psi \left( \intTor{ {\bf u} \cdot {\tn{G}}^N (\varrho, \varrho {\bf u}) } \right) {\rm d} W &+ \frac{1}{2} \int_0^T
\psi \bigg( \sum_{k \geq 1}
 \intTor{ \frac{ | {\bf G}_k (\varrho, \varrho {\bf u}) |^2 }{\varrho} } \bigg) {\rm d}t.
\end{split}
\end{equation}

\medskip

{\it Second approximation level:}

\medskip

Our goal is to let $N \to \infty$ in (\ref{ei2}). First, we modify the compactness argument of \cite[Subsection 4.1]{BrHo} as follows: Setting
\[
M_N(t)=\sum_{k\geq1}\int_0^t\intTor{ \vu \cdot {\tn{G}}^{N}(\varrho, \varrho \vu ) } \,\Dif W
\]
and $\mathcal{X}_M=C[0,T]$, we denote by $\mu_{M_N}$ the law of $M_N$. Due to the uniform estimates
obtained in  \cite{BrHo}, each process $M_N$ is a martingale and the set $\{ \mu_{M_N} \}_{N \geq 1}$ is tight on $\mathcal{X}_M$. Therefore we may include the sequence $\{ M_N\}_{N \geq 1}$ to the result of \cite[Proposition 4.5]{BrHo} to obtain, after the change of probability space, a new sequence $\{ \tilde M_N \}_{N \geq 1}$ having the same law as the original $\{ M_N \}_{N \geq 1}$ and converging to some $\tilde M$ a.s. in $\mathcal{X}_M$. Moreover, the space of continuous square integrable martingales is closed we deduce that the limit $\tilde M$ is also a martingale. Besides, it follows from the equality of joint laws that \eqref{ei2} is also satisfied on the new probability space.

Next, by virtue of hypotheses (\ref{FG1}), (\ref{FG2}), the function
\[
[\vr, \vc{q}] \mapsto \sum_{k \geq 1} \frac{ |{\bf G}_k (\varrho, {\bf q}) |^2 }{\varrho} \ \mbox{is continuous},
\]
and
\[
\sum_{k \geq 1} \frac{ |{\bf G}_k (\varrho, {\bf q}) |^2 }{\varrho} \leq c \left(  \varrho  + \frac{| {\bf q} |^2}{\vr}  \right)
\]
is sublinear in $\varrho$ and $|{\bf q}|^2/\vr$ and as such dominated by the total energy
\[
\frac{1}{2} \left( \frac{1}{2} \vr |\vu|^2 + H(\vr) \right) + 1.
\]

Thus following the arguments of \cite[Section 4]{BrHo} we may let $N \to \infty$ in (\ref{ei2}) to conclude
\begin{equation} \label{ei3}
\begin{split}
\int_0^T \psi \intTor{  \mathbb{S} (\nabla {\bf u}): \nabla {\bf u} } \ {\rm d}t &\leq \psi(0) \intTor{ \Big[
\frac{ |(\vr \vu)(0, \cdot) |^2 }{2\vr(0, \cdot)} + H_\delta(\varrho(0, \cdot)) \Big] }
\\
&+ \int_0^T \partial_t \psi
\left( \intTor{ \Big[ \frac{1}{2} \varrho | {\bf u} |^2 + H_\delta(\varrho) \Big] } \ \right) {\rm d}t
\\
+  \int_0^T  \psi\,  {\rm d} \tilde M &+ \frac{1}{2} \int_0^T
\psi \bigg( \sum_{k \geq 1}
 \intTor{ \frac{ | {\bf G}_k (\varrho, \varrho {\bf u}) |^2 }{\varrho} } \bigg) {\rm d}t.
\end{split}
\end{equation}

\medskip

{\it Third and fourth approximation level:}

\medskip

Repeating exactly the same arguments we may let successively $\ep \to 0$ and $\delta \to 0$ in (\ref{ei3}) to obtain (\ref{EI2}) thus proving
Proposition \ref{prop:energynew}

\subsection{Relative energy inequality - proof of Theorem \ref{thREI}}
\label{MEI}

We start with the following auxiliary result.

\begin{Lemma} \label{lem}

Let $s$ be a stochastic process on $\StoB$ such that for some $\lambda\in\R$,
\[
s \in C_{\rm weak}([0,T]; W^{-\lambda,2}(\tor)) \cap L^\infty (0,T; L^1(\tor)) \quad \text{$\mathbb{P}$-a.s.},
\]
\begin{equation} \label{hh1}
 \E\bigg[ \sup_{t \in [0,T]} \| s \|^p_{L^1(\tor)}\bigg]  < \infty \ \mbox{for all}\ 1 \leq p < \infty,
\end{equation}
\begin{equation} \label{rel1}
\Dif s= \DD s\, \dt + \DS s\,\Dif W.
\end{equation}
Here $\DD s, \DS s$ are progressively measurable with
\begin{align} \label{hh2}
\begin{aligned}
\DD s\in &L^p(\Omega;L^1(0,T;W^{-\lambda,q}(\tor)),\quad \DS s\in L^2(\Omega;L^2(0,T;L_2(\mathfrak U;W^{-m,2}(\tor)))),\\
&\sum_{k\geq 1}\int_0^T\|\DS s(e_k)\|^2_1\in L^p(\Omega)\quad 1\leq p<\infty,
\end{aligned}
\end{align}
for some $q>1$ and some $m\in\N$.

Let $r$ be a stochastic process on $\StoB$ satisfying
\[
r \in C([0,T]; W^{\lambda,q'} \cap C (\tor)) \quad \text{$\mathbb{P}$-a.s.},
\]
\begin{equation} \label{hh3}
\E\bigg[ \sup_{t \in [0,T]} \| r \|_{W^{\lambda,q'} \cap C (\tor)}^p \bigg] < \infty,\ 1 \leq p < \infty,
\end{equation}
\begin{equation} \label{rel2}
\Dif r = \DD r + \DS r \, \Dif W.
\end{equation}
Here $\DD r, \DS r$ are progressively measurable with
\begin{align} \label{hh4}
\begin{aligned}
\DD r\in L^p&(\Omega;L^1(0,T;W^{\lambda,q'}\cap C(\tor)),\quad \DS r\in L^2(\Omega;L^2(0,T;L_2(\mathfrak U;W^{-m,2}(\tor)))),\\
&\sum_{k\geq1} \int_0^T\|\DS r(e_k)\|^2_{W^{\lambda,q'}\cap C(\tor)}\dt\in L^p(\Omega)\quad 1\leq p<\infty.
\end{aligned}
\end{align}
Let $Q$ be $[\lambda+2]$-continuously differentiable function satisfying
\begin{equation} \label{hh5}
\E\bigg[\sup_{t \in [0,T]} \| Q^{(j)} (r) \|_{W^{\lambda,q'} \cap C (\tor)}^p \bigg]< \infty \quad j = 0,1,2,\quad 1 \leq p < \infty.
\end{equation}

Then
\begin{equation} \label{result}
\begin{split}
\Dif \left( \intTor{ s Q(r) } \right)
&= \bigg( \intTor{ \Big[  s  \Big( Q'(r) \DD r   + \frac{1}{2}\sum_{k\geq1} Q''(r)
\left| \DS r (e_k)\right|^2  \Big) \Big] }  +  \left< Q(r) , \DD s \right> \bigg) {\rm d}t
\\
&+ \bigg(  \sum_{k\geq1}\intTor{  \DS s(e_k)\,\DS r(e_k)  } \bigg) {\rm d}t
+ {\rm d}\tn{M},
\end{split}
\end{equation}
where
\begin{equation} \label{result1}
\tn{M} = \sum_{k\geq1}\int_0^t\intTor{ \Big[  s  Q'(r) \DS r(e_k)  + Q(r) \DS s(e_k)  \Big] }\,\Dif W_k.
\end{equation}
\end{Lemma}

{\bf Proof:}

\medskip

In accordance with hypothesis (\ref{hh3}),
relation (\ref{rel2}) holds point-wise in $\tor$. Consequently, we may apply It\^{o}'s chain rule
to obtain
\begin{equation} \label{step1}
\Dif Q(r) = Q'(r) \left[ \DD r  {\rm d}t + \DS r \,{\rm d}W \right] + \frac{1}{2} \sum_{k\geq1}Q''(r) \left|
 \DS r(e_k) \right|^2 {\rm d}t
\end{equation}
pointwise in $\tor$.

Next, we regularize (\ref{rel1}) by taking a spatial convolution with a suitable family of regularizing kernels. Denoting $[v]_\delta$ the regularization
of $v$, we may write
\[
\Dif [s]_\delta = \big[ \DD s \big]_\delta \ \dt + \big[ \DS s \big]_\delta \, \Dif W
\]
pointwise in $\tor$. Thus by It\^{o}'s product rule
\begin{equation} \label{step2}
\begin{split}
{\rm d} \Big( [ s ]_\delta  Q(r) \Big) &= \left[ s \right]_\delta {\rm d} Q(r) + Q(r) {\rm d} [s]_\delta +\sum_{k\geq1}
 [ \DS s  ]_\delta(e_k) \, \DS r(e_k)  \ {\rm d}t
\\
&= \bigg[ [ s ]_\delta \bigg( Q'(r) \DD r   + \frac{1}{2}\sum_{k\geq1} Q''(r)
\left| \DS r(e_k) \right|^2  \bigg)  + Q(r) \big[ \DD s \big]_\delta  \bigg] {\rm d}t
\\
&+ \Big[ [ s ]_\delta Q'(r)  \DS r   + Q(r) \left[  \DS s  \right]_\delta \Big] \,{\rm d} W
 +
\sum_{k\geq1}\left[ \DS s  \right]_\delta (e_k) \,\DS r(e_k)  \ {\rm d}t
\end{split}
\end{equation}
pointwise in $\tor$.
Integrating (\ref{step2}) we therefore obtain
\begin{equation} \label{step3}
\begin{split}
{\rm d} &\intTor{ [ s ]_\delta  Q(r) }
= \intTor{ \bigg[ [ s ]_\delta \bigg( Q'(r) \DD r   + \frac{1}{2}\sum_{k\geq1} Q''(r)
\left| \DS r(e_k) \right|^2  \bigg)  + Q(r) \left[ \DD s \right]_\delta  \bigg] } {\rm d}t
\\
&+ \intTor{ \Big[ [ s ]_\delta Q'(r)  \DS r   + Q(r) \left[  \DS s  \right]_\delta \Big] } \,{\rm d} W
 +
\sum_{k\geq1}\intTor{ \left[ \DS s  \right]_\delta(e_k)  \,  \DS r(e_k) } \ {\rm d}t.
\end{split}
\end{equation}

Finally, using hypotheses (\ref{hh1}), (\ref{hh2}), (\ref{hh3}), (\ref{hh4}), and (\ref{hh5})
we are able to
perform the limit $\delta \to 0$ in (\ref{step3}) completing the proof.

\qed

\begin{Remark}

The result stated in Lemma \ref{lem} is not optimal with respect to the regularity properties of the processes
$r$ and $s$. As a matter of fact, we could regularize both $r$ and $s$ in the above proof to conclude that
(\ref{result}) holds as long as all expressions in (\ref{result}), (\ref{result1}) are well defined.

\end{Remark}

\color{black}

Now, we are ready to complete the proof of the relative energy inequality (\ref{REI2}). We start by writing
\[
\begin{split}
\mathcal{E} \left( \varrho, \vu \ \Big| \ r, {\bf U} \right)
&=
\intTor{ \left[ \frac{1}{2} \vr |\vu|^2 + H(\varrho) \right] }
\\
&- \intTor{ \varrho \vu \cdot {\bf U} }
+ \intTor{ \frac{1}{2} \vr |{\bf U}|^2 } - \intTor{ \vr H'(r) } - \intTor{ \left[ H'(r) r - H(r) \right] }.
\end{split}
\]
As the time evolution of the first integral is governed by the energy inequality (\ref{EI2}), it remains to compute the time differentials
of the remaining terms with the help of Lemma \ref{lem}.

\medskip

\noindent {\bf Step 1:}

\medskip

To compute $\Dif \intTor{ \vr \vu \cdot \vc{U} }$ we recall that $s = \varrho \vu $ satisfies hypotheses (\ref{hh1}), (\ref{hh2})
with $l=1$ and some $q < \infty$.  Applying Lemma \ref{lem} we obtain
\begin{equation} \label{I1}
\begin{split}
\Dif \left( \intTor{ \vr \vu \cdot {\bf U} } \right) &= \left( \intTor{ \left[ \vr \left( \vu \cdot \DD {\bf U}
+ \vu \cdot \nabla {\bf U}  \cdot \vu \right) + \Div {\bf U} p(\vr) - \tn{S} (\Grad \vu ): \nabla {\bf U} \right] } \right)  {\rm d}t \\
& +  \sum_{k\geq1}\intTor{ \DS {\bf U}(e_k) \cdot\vc{G}_k (\varrho,\varrho\vu)}\, {\rm d}t+ \Dif M_1  ,
\end{split}
\end{equation}
where
\[
M_1(t) = \int_0^t \intTor{ \vc{U} \cdot \tn{G}(\vr, \vr \vu) } \,\Dif W + \int_0^t \intTor{ \vr \vu \cdot \DS \vc{U} } \,\Dif W
\]
is a square integrable martingale.

\color{black}

\medskip

\noindent
{\bf Step 2:}

\medskip

Similarly, we compute
\begin{equation} \label{I2}
\begin{split}
\Dif \left( \intTor{ \frac{1}{2} \vr |\vc{U}|^2 } \right) &=
\intTor{ \varrho \vu \cdot \Grad {\bf U} \cdot {\bf U} }  {\rm d}t\\
&+  \intTor{ \vr {\bf U} \cdot  \DD {\bf U} } {\rm d}t
+ \frac{1}{2} \sum_{k\geq1}\intTor{ \vr |\DS {\bf U}(e_k)|^2 } \ {\rm d}t + {\rm d}M_2,
\end{split}
\end{equation}
\[
M_2 = \int_0^t \intTor{ \vr \vc{U} \cdot \DS \vc{U} } \, \Dif W,
\]
\begin{equation} \label{I3}
{\rm d} \left( \intTor{ \left[ H'(r) r - H(r) \right] } \right)
= \intTor{ p'(r) \DD r } \ {\rm d}t + \frac{1}{2}\sum_{k\geq1} \intTor{ p''(r) |\DS r(e_k)|^2 } \ {\rm d}t + \Dif M_3,
\end{equation}
\[
M_3 = \int_0^t \intTor{ p'(r) \DS r } \,\Dif W,
\]
and, finally,
\begin{equation} \label{I4}
\begin{split}
{\rm d} \left( \intTor{ \varrho H'(r) } \right) &=
 + \intTor{ \varrho \Grad H'(r) \cdot \vu } \ {\rm d}t \\ &+
\intTor{ \vr H''(r) \DD r } \ {\rm d}t + \frac{1}{2}\sum_{k\geq1} \intTor{ \vr H'''(r) |\DS r(e_k) |^2 } \ {\rm d}t
+ \Dif M_4,
\end{split}
\end{equation}
\[
M_4(t) = \int_0^t \intTor{ \vr H''(r) \DS r } \,\Dif W.
\]

\medskip

\noindent
{\bf Step 3:}

Now, we can derive a ``differential form'' of (\ref{I1}--\ref{I4}) similar to (\ref{EI2}) by applying Lemma (\ref{lem}) to the product with
a test function $\psi$. Summing up the resulting expressions and adding the sum to (\ref{EI2}), we obtain (\ref{REI2}).
We have proved Theorem \ref{thREI}.

\section{Weak--strong uniqueness}
\label{WS}

As the first application of Theorem \ref{thREI} we present a weak-strong uniqueness result. To this end, let us introduce the following notion of strong solution to the stochastic Navier-Stokes system.

\color{black}

\begin{Definition}\label{def:strsol}

Let $\StoB$ be a stochastic basis with a complete right-continuous filtration, let ${W}$ be an $\left\{ \mathfrak{F}_t \right\}_{t \geq 0} $-cylindrical Wiener process. A pair
 $(\varrho,\vu)$ and a stopping time $\mathfrak{t}$ is called a (local) strong solution system \eqref{eq1}--\eqref{eq3} provided
\begin{itemize}
\item the density $\varrho > 0$ $\mathbb{P}$-a.s., $t \mapsto \varrho(t, \cdot) \in W^{3,2}(\tor)$ is $\left\{ \mathfrak{F}_t \right\}_{t \geq 0}$-adapted,
\[
\expe{  \sup_{t \in [0,T]} \| \varrho (t, \cdot) \|_{W^{3,2} (\tor)}^p } < \infty \ \mbox{for all}\ 1 \leq p < \infty;
\]
\item the velocity $t \mapsto \vu (t, \cdot) \in W^{3,2}(\tor;R^3)$ is $\left\{ \mathfrak{F}_t \right\}_{t \geq 0}$-adapted and,
\[
\expe{ \sup_{t \in [0,T]} \| \vu (t, \cdot) \|_{W^{3,2} (\tor;R^3)}^p } < \infty\  \mbox{for all}\ 1 \leq p < \infty;
\]
\item for all $t\in[0,T]$ there holds $\prst$-a.s.
\[
\begin{split}
\varrho (t\wedge\mathfrak{t}) &= \varrho(0) -  \int_0^{t \wedge \mathfrak{t}} \Div(\varrho\vu ) \ \dt \\
(\varrho \vu) (t\wedge\mathfrak{t})  &= (\varrho \vu) (0) - \int_0^{t \wedge \mathfrak{t}} \Div (\varrho\vu \otimes\vu ) \ \dt \\
& + \int_0^{t \wedge \mathfrak{t}} \Div \tn{S}(\Grad \vu) \ \dt
- \int_0^{t \wedge \mathfrak{t}}\Grad p(\varrho)\ \dt + \int_0^{t \wedge \mathfrak{t}} {\tn{G}}(\varrho,\varrho\vu ) \,\Dif W.
\end{split}
\]
\end{itemize}
\end{Definition}

\color{black}

\begin{Remark} \label{RR3}

To the best of our knowledge, there is no existence results for the stochastic compressible Navier-Stokes system in the class of strong solutions. The regularity hypotheses imposed in Definition \ref{def:strsol} are inspired by the deterministic case studied by Valli \cite{Vall1} and Valli, Zajaczkowski \cite{VAZA}.

\end{Remark}

\subsection{Pathwise weak--strong uniqueness}

\color{black}

We claim the following pathwise variant of the weak--strong uniqueness principle.

%
%
%
%


\begin{Theorem}\label{thm:uniq}
The pathwise weak-strong uniqueness holds true for system \eqref{eq1}--\eqref{eq3} in the following sense: let $\left[ (\Omega,\mf,(\mf_t),\prst),\varrho,\bfu , W \right]$ be a dissipative martingale solution to system \eqref{eq1}--\eqref{eq3} and let $(\tilde \varrho, \tilde\bfu)$ and a stopping time
$\mathfrak{t}$ be a strong solution of the same problem defined on the same stochastic basis with the same Wiener process and with the initial data
\[
\tilde \varrho(0, \cdot)  = \varrho(0, \cdot), \quad \tilde \varrho(0, \cdot) \tilde \vu (0, \cdot) = (\varrho \bfu) (0, \cdot) \quad \mbox{$\mathbb{P}$-a.s.},
\]
\begin{equation} \label{hyp1}
\varrho(0, \cdot) \geq \underline{\vr} > 0 \quad \mathbb{P}\mbox{-a.s.}
\end{equation}
Then $\varrho (\cdot\wedge \mathfrak{t}) = \tilde \varrho (\cdot \wedge \mathfrak{t})$ and $\varrho \bfu (\cdot \wedge \mathfrak{t}) =
\tilde \varrho \tilde{\bf u} (\cdot \wedge \mathfrak{t})$ a.s.

\end{Theorem}

\color{black}

{\bf Proof of Theorem \ref{thm:uniq}:}

\noindent
{\bf Step 1:}

\medskip

We start by introducing a stopping time
\[
\tau_M = \inf \{ t \in (0,T) \ \Big| \  \ \left\| \tilde{\vu} (s, \cdot) \|_{W^{3,2}(\tor; R^3)} > M \right\}.
\]
As $(\tilde \varrho, \tilde \vu)$ is a strong solution,
\[
\mathbb{P} \left[ \lim_{M \to \infty} \tau_M = T \right] = 1;
\]
whence it is enough to show the result for a fixed $M$.

\medskip

\noindent
{\bf Step 2:}

\medskip

Given $M > 0$, we get, as a direct consequence of the embedding relation $W^{2,2}(\tor) \hookrightarrow C(\tor)$,
\[
\sup_{t \in [0, \tau_M]} \| \Grad \tilde{\vu} \|_{L^\infty(\tor; R^{3 \times 3})} \leq c(M),
\]
and, as $\tilde{\varrho}$ satisfies the equation of continuity on the time interval $[0, \mathfrak{t}]$
\color{black} and hypothesis (\ref{hyp1}),
\[
0 < \underline{\varrho}_M \leq \tilde \varrho (t \wedge \mathfrak{t}) \leq \overline{\varrho}_M \ \mbox{for}\ t \in [0, \tau_M].
\]

Next, it is easy to check that for any $\delta > 0$ (small enough)
\begin{equation} \label{coerc}
H(\varrho) - H'(r)(r)(\varrho - r) - H(r) \geq c(\delta) \left\{ \begin{array}{l} |\varrho - r|^2 \ \mbox{for any} \delta < r, \varrho  < \delta^{-1},
\\ \\
1 + \varrho^\gamma \ \mbox{whenever}\ \delta < r  < \delta^{-1}, \ \varrho \in (0, \infty) \setminus [ \delta/2, 2 \delta]. \end{array} \right.
\end{equation}
This motivates
the following definition. For
\[
\Phi_M \in C^\infty_0(0, \infty), \ 0 \leq \Phi_M \leq 1, \ \Phi(r) = r \ \mbox{for all}\ r \in [\underline{\varrho}_M/2, 2 \overline{\varrho}_M ],
\]
we introduce
\[
[h]_{\rm ess} = \Phi_M (\varrho) h ,\ [h]_{\rm res} = h - \Phi_M (\varrho) h \ \mbox{for any}\ h \in L^1(\Omega \times (0,T) \times \tor).
\]

It follows from (\ref{coerc}) that
\begin{equation} \label{coercess}
\mathcal{E} \left( \varrho, \ \vu \ \Big| \tilde \varrho, \ \tilde \vu \right)  \geq
c(M) \left[ \left\| \left[ \vu - \tilde \vu \right]_{\rm ess} \right\|^2_{L^2(\tor;R^3)} +
\left\| \left[ \varrho - \tilde \varrho \right]_{\rm ess} \right\|^2_{L^2(\tor)} \right],
\end{equation}
and similarly
\begin{equation} \label{coercres}
\mathcal{E} \left( \varrho, \ \vu \ \Big| \tilde \varrho, \ \tilde \vu \right)  \geq
c(M) \left[ \left\| \sqrt{\varrho} \left[ \vu - \tilde \vu \right]_{\rm res} \right\|^2_{L^2(\tor;R^3)} +
\left\| \left[ 1 + \varrho^\gamma \right]_{\rm res} \right\|_{L^1(\tor)} \right].
\end{equation}
whenever $t \in [0, \tau_M]$.

\medskip

\noindent
{\bf Step 3:}

\medskip

Our goal now is to apply the relative energy inequality (\ref{REI2}) to $r = \tilde \varrho$, ${\bf U} = \tilde \vu$ on the time interval
$[0, \tau_M \wedge \mathfrak{t} ]$. To this end, we compute
\[
{\rm d} \tilde \vu = {\rm d} \left( \frac{\tilde \varrho \tilde \vu}{\tilde \varrho} \right) = \frac{1}{\tilde \varrho} {\rm d} (\tilde \varrho
\tilde \vu) - \frac{\partial_t \tilde \varrho}{\tilde \varrho} \tilde \vu \ {\rm d}t;
\]
whence we can deduce from (\ref{REI2}) that
\begin{equation}
\label{REI1}
\begin{split}
\mathcal{E} \left( \varrho, \vu \Big| \tilde \varrho, \tilde \vu \right)(t \wedge \tau_M \wedge \mathfrak{t})  &+ \int_0^{t \wedge \tau_M\wedge \mathfrak{t}} \intTor{
\left( \tn{S} (\nabla \vu) - \tn{S} (\Grad  \tilde \vu) \right)
: \left( \Grad \vu - \Grad \tilde \vu \right) } \ {\rm d}s \\
&\leq M(t \wedge \tau_M \wedge \mathfrak{t}) - M(0) + \int_0^{t \wedge \tau_M \wedge \mathfrak{t}} \mathcal{R} \left( \varrho, \vu \Big| \tilde \varrho , \tilde \vu \right) \ {\rm d}t,
\end{split}
\end{equation}
\color{black}
with
\begin{align}
\mathcal{R}  \left( \varrho, \vu \Big| \tilde \varrho , \tilde \vu \right) &=\intTor{ \tn{S} (\Grad \tilde \vu):(\Grad \tilde \vu-\Grad \vu ) }
\nonumber\\
& - \intTor{ \frac{\varrho}{\tilde \varrho} \Big( \partial_t \tilde \varrho \tilde \vu + \Div (\tilde \varrho  \tilde \vu \otimes
\tilde \vu ) \Big) \cdot (\tilde \vu -\vu) }\nonumber \\
&+\intTor{ \varrho \vu \cdot\Grad \tilde \vu (\tilde \vu -\vu) } + \intTor{ \frac{\varrho}{\tilde \varrho} \Big( \Div \tn{S} ( \Grad \tilde \vu) -
\Grad p(\tilde \varrho) \Big) \cdot (\tilde \vu -\vu )}\nonumber \\
&+\intTor{ \big((\tilde{\varrho} -\varrho)H''(\tilde \varrho) \partial_t \tilde \varrho +\Grad H'(\tilde \varrho)(\tilde \varrho \tilde \vu-
\varrho \vu)\big)} -\intTor{ \Div\tilde \vu (p(\varrho)-p(\tilde \varrho)) } \nonumber\\
&+\frac{1}{2}\sum_{k\geq1} \intTor{  \varrho \Big|\frac{1}{{\varrho}} \vc{G}_k (\varrho,\varrho\vu) - \frac{1}{\tilde \varrho}
 \vc{G}_k (\tilde \varrho, \tilde \varrho \tilde \vu )\Big|^2 }\nonumber\\
&=\intTor{ \frac{1}{\tilde \varrho} (\varrho - \tilde \varrho) \Div \tn{S} (\Grad \tilde \vu) \cdot (\tilde \vu- \vu ) }\nonumber
\\
&+\intTor{ \varrho ( \vu - \tilde \vu) \cdot\Grad \tilde \vu \cdot (\tilde \vu -\vu) } - \intTor{ \frac{\varrho}{\tilde \varrho}
\Grad p(\tilde \varrho) \cdot (\tilde \vu -\vu) } \nonumber\\
&+\intTor{ \big((\tilde{\varrho} -\varrho)H''(\tilde \varrho) \partial_t \tilde \varrho +\Grad H'(\tilde \varrho)(\tilde \varrho \tilde \vu-
\varrho \vu)\big)} -\intTor{ \Div \tilde \vu (p(\varrho)-p(\tilde \varrho)) } \nonumber\\
&+\frac{1}{2} \sum_{k\geq1}\intTor{  \varrho \Big|\frac{1}{{\varrho}} \vc{G}_k (\varrho,\varrho\vu) - \frac{1}{\tilde \varrho}
 \vc{G}_k (\tilde \varrho, \tilde \varrho \tilde \vu )\Big|^2 }\nonumber\\
&=\intTor{ \frac{1}{\tilde \varrho} (\varrho - \tilde \varrho) \Div \tn{S} (\nabla\tilde \vu ) \cdot (\tilde \vu- \vu ) }
+\intTor{ \varrho ( \vu - \tilde \vu) \cdot\Grad \tilde \vu \cdot (\tilde \vu -\vu) } \nonumber\\
&-\intTor{ \Div\tilde \vu \Big(p(\varrho)- p'(\tilde \varrho) (\varrho - \tilde \varrho ) - p(\tilde \varrho) \Big) }
+\frac{1}{2}\sum_{k\geq1}\intTor{ \varrho \Big|\frac{1}{{\varrho}}  \vc{G}_k (\varrho,\varrho\vu) - \frac{1}{\tilde \varrho}
 \vc{G}_k (\tilde \varrho, \tilde \varrho \tilde \vu)\Big|^2 }\nonumber\\
&=\mathscr T_1+\mathscr T_2+\mathscr T_3+\mathscr T_4.\label{rem1}
\end{align}
The goal is to estimate the terms $\mathscr T_1,...,\mathscr T_4$ and to absorb them in the left-hand-side of \eqref{REI1} via Gronwall's lemma. Repeating the estimates from \cite{FENOSU}, we deduce that
\begin{align}\label{eq:T12}
\mathscr T_1+\mathscr T_2+\mathscr T_3\leq c(M)\mathcal E\Big([\varrho,\bfu]\Big|[\tilde \varrho,\bfU]\Big).
\end{align}
Now we estimate the part arising from the correction term and decompose
\begin{align*}
\mathscr T_4&=\frac{1}{2} \,\sum_k\int_{\mt}\chi_{\varrho\leq\frac{\tilde \varrho}{2}}\varrho\Big(\frac{\vc{G}_k(\varrho,\varrho\bfu)}{\varrho}-\frac {\vc{G}_k(\tilde \varrho,\tilde \varrho\bfU)}{\tilde \varrho}\Big)^2\dx\\
&+\frac{1}{2}\,\sum_k\int_{\mt}\chi_{\frac{\tilde \varrho}{2}\leq \varrho\leq 2\tilde \varrho}\varrho\Big(\frac{\vc{G}_k(\varrho,\varrho\bfu)}{\varrho}-\frac {\vc{G}_k(\tilde \varrho,\tilde \varrho\bfU)}{\tilde \varrho}\Big)^2\dx\\
&+\frac{1}{2}\,\sum_k\int_{\mt}\chi_{\varrho\geq 2\tilde \varrho}\varrho\Big(\frac{\vc{G}_k(\varrho,\varrho\bfu)}{\varrho}-\frac {\vc{G}_k(\tilde \varrho,\tilde \varrho\bfU)}{\tilde \varrho}\Big)^2\dx\\
&=\mathscr T_4^1+\mathscr T_4^2+\mathscr T_4^3.
\end{align*}

Using \eqref{FG1}, \eqref{FG2} and \eqref{coerc} there holds
\begin{align*}
\mathscr T_4^1&\leq \,c(M)\,\int_{\mt}\chi_{\varrho\leq\frac{\tilde \varrho}{2}}(1+\varrho|\bfu|^2+\varrho|\bfU|^2)\dx\\
&\leq \,c(M)\,\int_{\mt}\chi_{\varrho\leq\frac{\tilde \varrho}{2}}\dx+ \,c(M)\,\E\int_{\mt}\varrho|\bfu-\bfU|^2\dx\\
&\leq \,c(M)\,\int_{\mt}\chi_{\varrho\leq\frac{\tilde \varrho}{2}}\big(H(\varrho)-H'(\tilde \varrho)(\varrho-\tilde \varrho)-H(\tilde \varrho)\big)\dx+ \,c(M)\,\int_{\mt}\varrho|\bfu-\bfU|^2\dx\\
&\leq\,c(M)\,\mathcal E\Big([\varrho,\bfu]\Big|[\tilde \varrho,\bfU]\Big)\Big].
\end{align*}
Similarly we gain by \eqref{coerc} and the mean-value theorem
\begin{align*}
\mathscr T_4^2&\leq\,\frac{1}{2}\,\sum_{k\geq1}\int_{\mt}\chi_{\frac{\tilde \varrho}{2}\leq \varrho\leq 2\tilde \varrho}\varrho\Big(\frac{\vc{G}_k(\varrho,\varrho\bfu)}{\varrho}-\frac {\vc{G}_k(\tilde \varrho,\varrho\bfu)}{r}\Big)^2\dx\\&+\frac{1}{2}\,\sum_k\int_{\mt}\chi_{\frac{\tilde \varrho}{2}\leq \varrho\leq 2\tilde \varrho}\varrho\Big(\frac{\vc{G}_k(\tilde \varrho,\varrho\bfu)}{\tilde \varrho}-\frac {\vc{G}_k(\tilde \varrho,\tilde \varrho\bfU)}{\tilde \varrho}\Big)^2\dx\\
&\leq \,c(M)\,\int_{\mt}\chi_{\frac{\tilde \varrho}{2}\leq \varrho\leq 2\tilde \varrho}\Big(|\varrho-\tilde \varrho|^2(1+|\varrho\bfu|^2)+|\varrho\bfu-\tilde \varrho\bfU|^2\Big)\dx\\
&\leq \,c(M)\,\int_{\mt}\chi_{\frac{\tilde \varrho}{2}\leq \varrho\leq 2\tilde \varrho}\Big(|\varrho-\tilde \varrho|^2(1+|\bfU|^2)+|\varrho(\bfu-\bfU)|^2\Big)\dx\\
&\leq\,c(M)\,\int_{\mt}\chi_{\frac{\tilde \varrho}{2}\leq \varrho\leq 2r}|\varrho-\tilde \varrho|^2\dx+\,\int_{\mt}\varrho|\bfu-\bfU|^2\dx\\
&\leq \,c(M)\,\int_{\mt}\big(H(\varrho)-H'(\tilde \varrho)(\varrho-\tilde \varrho)-H(\tilde \varrho)\big)\dx+\,\mathcal E\Big([\varrho,\bfu]\Big|[\tilde \varrho,\bfU]\Big)\\
&\leq\,c(M)\,\mathcal E\Big([\varrho,\bfu]\Big|[\tilde \varrho,\bfU]\Big)
\end{align*}
Finally, \eqref{coerc} yields
\begin{align*}
\mathscr T_4^3&\leq \,c(M)\,\int_{\mt}\chi_{ \varrho\geq 2\tilde\varrho}\Big(\varrho+\varrho|\bfu|^2+\varrho|\bfU|^2\Big)\dx\\
&\leq \,c(M)\,\int_{\mt}\chi_{ \varrho\geq 2\tilde \varrho}\Big(\varrho+\varrho|\bfu-\bfU|^2+\varrho|\bfU|^2\Big)\dx\\
&\leq \,c(M)\,\int_{\mt}\chi_{ \varrho\geq 2\tilde \varrho}\Big(\varrho^\gamma(1+|\bfU|^2)+\varrho|\bfu-\bfU|^2\Big)\dx\\
&\leq \,c(M)\,\int_{\mt}\big(H(\varrho)-H'(\tilde \varrho)(\varrho-r)-H(r)\big)\dx+\,\mathcal E\Big([\varrho,\bfu]\Big|[\tilde \varrho,\bfU]\Big)\\
&\leq\,c(M)\,\mathcal E\Big([\varrho,\bfu]\Big|[\tilde \varrho,\bfU]\Big).
\end{align*}

Plugging everything together we deduce that
\[
\mathcal{E} \left( \varrho, \vu \Big| \tilde \varrho, \tilde \vu \right)(t \wedge \tau_M \wedge \mathfrak{t})
\leq M(t \wedge \tau_M \wedge \mathfrak{t}) - M(0) + c(M) \int_0^{t \wedge \tau_M \wedge \mathfrak{t}} \mathcal{E} \left( \varrho, \vu \Big| \tilde \varrho , \tilde \vu \right) \ {\rm d}t.
\]
Averaging over $\Omega$ and applying Gronwall's lemma we conclude the proof.

\qed

\subsection{Weak--strong uniqueness in law}

Strictly speaking, the strong and weak martingale solutions of problem (\ref{eq1}--\ref{eq3}) may not be defined on the same
probability space and with the same Wiener process $W$. As a consequence of Theorem \ref{thm:uniq} we obtain the weak-strong uniqueness in law.

\begin{Theorem}\label{thm:uniqlaw}
The weak-strong uniqueness in law holds true. That is, if
$$\left[ (\Omega^1,\mf^1,(\mf^{1}_{t}),\prst^1), \vr^1 ,\vu^1 , W^1 \right]$$
is a dissipative martingale solution to system \eqref{eq1}--\eqref{eq3} and
$$\left[ (\Omega^2,\mf^2,(\mf^{2}_{t}),\prst^2), \vr^2 ,\vu^2 , W^2\right]$$
is a strong martingale solution of the same problem such that
$$\Lambda=\prst^1\circ (\varrho^1(0),\varrho^1\bfu^1(0))^{-1}=\prst^2\circ (\varrho^2(0),\varrho^2\bfu^2(0))^{-1},$$
then
\begin{equation}\label{law}
\prst^1\circ(\varrho^1,\varrho^1\bfu^1)^{-1}=\prst^2\circ(\varrho^2,\varrho^2\bfu^2)^{-1}.
\end{equation}
\end{Theorem}

\begin{proof}
The proof is based on the ideas of the classical result of Yamada--Watanabe for SDEs as presented for instance in \cite[Proposition 3.20]{karatzas}, however, we need to face several substantial difficulties that originate in the complicated structure of system \eqref{eq1}--\eqref{eq3}.

Let $R^1:=\varrho^1-\varrho^1(0)$, $R^2:=\varrho^2-\varrho^2(0)$, $\bfQ^1:= \varrho^1\vu^1-(\varrho^1\vu^1)(0)$, $\bfQ^2:=\vr^2\vu^2-(\varrho^2\bfu^2)(0)$. Let $M^1$ be the real-valued martingale from the energy inequality \eqref{EI2} of the dissipative solution $(\varrho^1,\varrho^1\bfu^1)$ and let $M^2\equiv0$. Set
\begin{align*}
\Theta:&=L^\gamma_x\times L^{\frac{2\gamma}{\gamma+1}}_x\times C([0,T];\mathfrak{U}_0)\times C([0,T];\mr)\\
&\qquad\times C_w([0,T]; L^\gamma_x)\times  C_w([0,T];L^\frac{2\gamma}{\gamma+1}_x)\times L^2(0,T;W^{1,2}_x)
\end{align*}
We denote by $\theta=(r_0,\bfq_0,w,m,r,\bfq,\bfv)$ a generic element of $\Theta$. Let $\mathcal{B}_T(\Theta)$ denote the $\sigma$-field on $\Theta$ given by
\begin{align*}
\mathcal{B}_T(\Theta):&=\mathcal{B}(L^\gamma_x)\otimes\mathcal{B}\big(L^\frac{2\gamma}{\gamma+1}_x\big)\otimes\mathcal{B}\big(C([0,T];\mathfrak{U}_0)\big)\otimes \mathcal{B}(C([0,T];\mr))\\
&\qquad\otimes\mathcal{B}_T\big(C_w([0,T];L^\gamma_x)\big)\otimes\mathcal{B}_T\big(C_w([0,T];L^\frac{2\gamma}{\gamma+1}_x)\big)\otimes\mathcal{B}(L^2(0,T;W^{1,2}_x)),
\end{align*}
where for a separable Banach space $X$ we denote by $\mathcal{B}(X)$ its Borel $\sigma$-field and by $\mathcal{B}_T(C_w([0,T];X))$ the $\sigma$-field generated by the mappings
$$C_w([0,T];X)\to X,\quad h\mapsto h(s),\qquad s\in [0,T].$$
The discussion in \cite[Section 3]{onsebr} shows that $(C_w([0,T];X),\mathcal{B}_T(C_w([0,T];X))$ is a Radon space, i.e. every probability measure on $(C_w([0,T];X),\mathcal{B}_T(C_w([0,T];X))$ is Radon. Since the same is true for any Polish space equipped with the Borel $\sigma$-field and since the topological product of a countable collection of Radon spaces is a Radon space, we deduce that $(\Theta,\mathcal{B}_T(\Theta))$ is a Radon space. Due to \cite[Theorem 3.2]{LFR}, every Radon space enjoys the regular conditional probability property. Namely, if $P$ is a probability measure on $(\Theta,\mathcal{B}_T(\Theta))$, $(E,\mathcal{E})$ is a measurable space and
$$T:(\Theta,\mathcal{B}_T(\Theta),P)\to (E,\mathcal{E})$$
is a measurable mapping, then there exists a regular conditional
probability with respect to $T$: that is, there exists is a function $K : E \times\mathcal{B}_T(\Theta)\to[0, 1]$, called a transition probability, such that
\begin{enumerate}
\item[(i)] $K(x,\cdot)$ is a probability measure on $\mathcal{B}_T(\Theta)$, for all $x\in E$,
\item[(ii)] $K(\cdot, A)$ is a measurable function on $(E,\mathcal{E})$, for all $A\in\mathcal{B}_T(\Theta)$,
\item[(iii)] for all $A\in \mathcal{B}_T(\Theta)$ and all $ B\in \mathcal{E}$ it holds true
$$P\big(A\cap T^{-1}(B)\big)=\int_B K(x,A)\,(T_\ast P)(\dif x),$$
where $T_\ast P$ denotes the pushforward measure on $(E,\mathcal{E})$.
\end{enumerate}

Let $j\in\{1,2\}$ and let $\mu^j$ denote the joint law of $(\vr^j(0),(\vr^j\vu^j)(0),W^j,M^j,R^j,\bfQ^j,\bfu^j)$ on $\Theta$, let $\p^W$ be the Wiener measure on $C([0,T];\mathfrak{U}_0)$ which also coincides with the projection to $w$ of $\mu^j$. The law of $(r_0,\bfq_0)$ is $\Lambda$ and the law of $(r_0,\bfq_0,w)$ is the product measure $\Lambda\otimes\p^W$ since $(\varrho^j(0),(\varrho^j\bfu^j)(0))$ is $\mf^j_0$-measurable and $W^j$ is independent of $\mf^j_0$. Furthermore,
$$\mu^j\big[(r(0),\bfq(0))=0\big]=1.$$

Now, we have all in hand to bring the two solutions $(\varrho^1,\bfu^1,W^1)$ and $(\varrho^2,\bfu^2,W^2)$ to the same probability space while preserving their joint laws. To this end, we recall that on $(\Theta,\mathcal{B}_T(\Theta),\mu^j)$ there exists a regular conditional probability with respect to $(r_0,\bfq_0,w)$, denoted by $K^j$.
Besides, since $\Theta$ is a product space and $(r_0,\bfq_0,w)$ is the projection to the first three coordinates, we may regard $K^j$ as a function on
\begin{align*}
&\Big[L^\gamma_x\times L^\frac{2\gamma}{\gamma+1}_x\times C([0,T];\mathfrak{U}_0)\Big]\\
&\quad\times\Big[\mathcal{B}(C([0,T];\mr))\otimes\mathcal{B}_T\big(C_w([0,T];L^\gamma_x)\big)\otimes\mathcal{B}_T\big(C_w([0,T];L^\frac{2\gamma}{\gamma+1}_x)\big)\otimes\mathcal{B}(L^2(0,T;W^{1,2}_x))\Big]
\end{align*}
and the property (iii) above rewrites as follows: let
$$A_1\in \mathcal{B}(L^\gamma_x)\otimes\mathcal{B}(L^\frac{2\gamma}{\gamma+1}_x)\otimes\mathcal{B}(C([0,T];\mathfrak{U}_0))$$
and
$$A_2\in \mathcal{B}(C([0,T];\mr))\otimes\mathcal{B}_T\big(C_w([0,T];L^\gamma_x)\big)\otimes\mathcal{B}_T\big(C_w([0,T];L^\frac{2\gamma}{\gamma+1}_x)\big)\otimes\mathcal{B}(L^2(0,T;W^{1,2}_x)),$$
then
\begin{equation}\label{iii}
\mu^j\big[A_1\times A_2\big]=\int_{A_1}K^j(r_0,\bfq_0,w,A_2)\Lambda\big(\dd (r_0,\bfq_0)\big)\p^W(\dd w).
\end{equation}
Finally, we define
$$\Omega:=\Theta\times C([0,T];\mr)\times C_w([0,T]; L^\gamma_x)\times  C_w([0,T];L^\frac{2\gamma}{\gamma+1}_x)\times L^2(0,T;W^{1,2}_x) $$
and denote by $\mathfrak{F}$ the $\sigma$-field on $\Omega$ given as the completion of
$$\mathcal{B}_T(\Theta)\otimes \mathcal{B}(C([0,T];\mr))\otimes\mathcal{B}_T\big(C_w([0,T];L^\gamma_x)\big)\otimes\mathcal{B}_T\big(C_w([0,T];L^\frac{2\gamma}{\gamma+1}_x)\big)\otimes\mathcal{B}(L^2(0,T;W^{1,2}_x))$$
with respect to the probability measure
\begin{equation}\label{3.23}
\p(\dd \omega):=K^1\big(r_0,\bfq_0,w,\dd (m_1,r_1,\bfq_1,\bfv_1)\big)K^2\big(r_0,\bfq_0,w,\dd (m_2,r_2,\bfq_2,\bfv_2)\big)\Lambda\big(\dd (r_0,\bfq_0)\big)\p^W(\dd w),
\end{equation}
where we have denoted by $\omega=(r_0,\bfq_0,w,m_1,r_1,\bfq_1,\bfv_1,m_2,r_2,\bfq_2,\bfv_2)$ a canonical element of $\Omega$. In order to endow $(\Omega,\mathfrak{F},\p)$ with a filtration that satisfies the usual conditions, we take
$$\mathfrak G_t:=\sigma\big((r_0,\bfq_0,w(s),m_1(s),r_1(s),\bfq_1(s),\bfv_1(s),m_2(s),r_2(s),\bfq_2(s),\bfv_2(s));\, 0\leq s\leq t\big),$$
$$\tilde{\mathfrak{G}}_t:=\sigma\big(\mathfrak{G}_t\cup \{N;\,\p(N)=0\}\big),\qquad\mathfrak{F}_t:=\bigcap_{\varepsilon\in(0,T-t)}\tilde{\mathfrak{G}}_{t+\varepsilon},\quad t\in[0,T).$$
Then due to \eqref{3.23} and \eqref{iii} it follows that
\begin{align*}
\p\big[\omega\in\Omega;&\,(r_0,\bfq_0,w,m_j,r_j,\bfq_j,\bfv_j)\in A_1\times A_2\big]\\
&=\int_{A_1\times A_2} K^j\big(r_0,\bfq_0,w,\dd (m_j,r_j,\bfq_j,\bfv_j)\big)\Lambda\big(\dd (r_0,\bfq_0)\big)\p^W(\dd w)\\
&=\int_{A_1} K^j\big(r_0,\bfq_0,w,A_2\big)\Lambda\big(\dd (r_0,\bfq_0)\big)\p^W(\dd w)\\
&=\mu^j\big[A_1\times A_2\big]\\
&=\p^j\big[(\vr^j(0),(\vr^j\vu^j)(0),W^j,M^j,R^j,\bfQ^j,\bfu^j)\in A_1\times A_2\big]
\end{align*}
hence the law of $(r_0,\bfq_0,w,m_j,r_j,\bfq_j,\bfv_j)$ under $\p$ coincides with the law of
$$(\vr^j(0),(\vr^j\vu^j)(0),W^j,M^j,R^j,\bfQ^j,\bfu^j)$$
under $\p^j$ and, as a consequence, the law of $(r_0+r_j,\bfq_0+\bfq_j,\bfv_j,w,m_j)$ under $\p$ coincides with the law of $(\varrho^j,\varrho^j\bfu^j,\bfu^j,W^j,M^j)$ under $\p^j$. In particular, $w$ is an $(\mathfrak{F}_t)$-cylindrical Wiener process.

To summarize, we have defined a stochastic basis $(\Omega,\mathfrak{F},(\mathfrak{F}_t),\p)$ with random variables $(r_0+r_j,\bfq_0+\bfq_j,\bfv_j,w)$ that have the same law as the original solutions $(\varrho^j,\varrho^j\bfu^j,\bfu^j,W^j)$, $j=1,2$. As a consequence,
$$\p\big[\bfq_0+\bfq_j=(r_0+r_j)\bfv_j\big]=1$$
and $(r_0+r_j,\bfq_0+\bfq_j,\bfv_j,w)$ solves \eqref{eq1}--\eqref{eq3} in the weak sense. This can be verified for instance by the method of \cite[Proposition 4.11]{BrHo}. Besides, since the law of $(\varrho^2,\bfu^2)$ is actually supported on a space of functions with higher regularity (see Definition \ref{def:strsol}) and $\varrho^2>0$, we deduce that $(r_0+r_2,\bfv_2,w)$ is a strong solution to \eqref{eq1}--\eqref{eq3}.

By the same reasoning as in Remark \ref{first} we obtain the following version of the energy inequality \eqref{EI2} which holds true for all $0\leq t\leq T$,
a.a. $0 \leq s \leq t$ including $s= 0$
$\p^1$-a.s.
\begin{equation*} \label{EI2a}
\begin{split}
&
 \intTor{ \Big[ \frac{1}{2} \varrho^1 | {\bf u}^1 |^2 + H(\varrho^1) \Big](t) }
+ \int_s^t \intTor{  \mathbb{S} (\nabla {\bf u}^1): \nabla {\bf u}^1 } \ {\rm d}r \\
\leq & \intTor{ \Big[ \frac{| (\vr^1 \vu^1)(s) |^2 }{2 \varrho^1(s)}  + H(\varrho^1(s)) \Big] }
+ \frac{1}{2} \int_s^t
\bigg(
\intTor{ \sum_{k \geq 1} \frac{ | {\bf G}_k (\varrho^1, \varrho^1 {\bf u}^1) |^2 }{\varrho^1} } \bigg) {\rm d}r
\\ &+ M^1(t)-M^1(s)
\end{split}
\end{equation*}
hence the equality of joint laws of $(r_0+r_1,\bfq_0+\bfq_1,\bfv_1,m_1)$ and $(\varrho^1,\varrho^1\bfu^1,\bfu^1,M^1)$ implies the corresponding inequality satisfied by $(r_0+r_1,\bfq_0+\bfq_1,\bfv_1,m_1)$. Since in view of Remark \ref{first} this is exactly the version of \eqref{EI2} that is used in the proof of pathwise weak--strong uniqueness,
%
%
%
%
Theorem \ref{thm:uniq} then applies and yields
$$\p\big[r_0+r_1=r_0+r_2,\;\bfq_0+\bfq_1=\bfq_0+\bfq_1\big]=1$$
or equivalently
$$\p\big[\omega=(r_0,\bfq_0,\bfw,m_1,r_1,\bfq_1,\bfv_1,m_2,r_2,\bfq_2,\bfv_2)\in\Omega;\,r_1=r_2,\,\bfq_1=\bfq_2\big]=1.$$
Hence, for all $A\in\mathcal{B}_T(C_w([0,T];L^\gamma_x))\otimes\mathcal{B}_T(C_w([0,T];L^\frac{2\gamma}{\gamma+1}_x))$,
\begin{align*}
\p^1\big[(\varrho^1,\varrho^1\bfu^1)\in A\big]&=\p\big[\omega\in \Omega;\, (r_0+r_1,\bfq_0+\bfq_1)\in A\big]\\
&=\p\big[\omega\in \Omega;\, (r_0+r_2,\bfq_0+\bfq_2)\in A\big]\\
&=\p^2\big[(\varrho^2,\varrho^2\bfu^2)\in A\big]
\end{align*}
and \eqref{law} follows.

\end{proof}

\color{black}

\section{Incompressible-inviscid limit}
\label{II}

As the second application of the relative energy inequality, we examine the inviscid, incompressible limit for the system
\begin{eqnarray}
  \Dif \vr + \Div (\vr \vu) \ \dt  &=& 0 \label{seq1} \\
  \label{seq2} \Dif (\vr \vu) + \left[ \Div (\vr \vu \otimes \vu) + \frac{1}{\ep^2} \Grad p(\vr) \right] \Dif t &=& \Div \mathbb{S}_\ep(\Grad \vu) \ \dt + \mathbb{G}(\vr, \vr \vu) \,\Dif W \\
  \tn{S}_\ep (\Grad \vu) &=& \mu_\ep \left(\Grad \vu + \Grad^t \vu -  \frac{2}{3} \Div \vu \mathbb{I} \right) + \eta_\ep \Div \vu \tn{I},\label{seq3},
\end{eqnarray}
where
\[
\mu_\ep, \ \eta_\ep \to 0 \ \mbox{as}\ \ep \to 0.
\]

The scaling in (\ref{seq1}--\ref{seq3}) reflects the situation when the Mach number is low and the Reynolds number is high, meaning the fluid is in a highly turbulent almost incompressible regime, see e.g. Klein et al.  \cite{KBSMRMHS}.
Under these circumstances, the motion is expected to be governed by the incompressible Euler system
\begin{eqnarray}
  \Div \vc{v} &=& 0 \label{Eeq1} \\
  \label{Eeq2} \Dif \vc{v} + \left[ \vc{v} \cdot \Grad \vc{v} + \Grad \Pi \right] \Dif t &=& \mathbb{G}(1, \vc{v}) \,\Dif W.
\end{eqnarray}

To compare the primitive and limit systems, we need that
\begin{itemize}
\item
the Navier-Stokes system (\ref{seq1}--\ref{seq3}) possesses a dissipative martingale solution
\[
\left[ \StoB; \vr, \vu, W \right],
\]
and the Euler system (\ref{Eeq1}), (\ref{Eeq2}) a (strong) solution on the same probability space $\StoB$ and
with the same Wiener process $W$;
\item
both $\vc{v}$ and the pressure $\Grad \Pi$ are smooth enough in the $x-$variable so that
$r = 1$, $\vc{U} = \vc{v}$ can be taken as test functions in the relative energy inequality (\ref{REI2}).
\end{itemize}
We address these issue in the following two sections.

\subsection{Solutions of the Navier-Stokes system}
\label{sns}

Given the initial data
\[
\vr_{0,\ep} \in L^\gamma(\tor), \ (\vr \vu)_{0,\ep} \in L^{\frac{2 \gamma}{\gamma + 1}}(\tor; R^3),
\]
with the associated law $\Lambda_\ep$ satisfying the hypotheses of Theorem \ref{thm:exist}
problem (\ref{seq1}--\ref{seq3}) admits a dissipative martingale solution
\[
\left[ \left(\Omega^\ep, \mathfrak{F}^\ep,\left\{\mathfrak{F}_t^\ep \right\}_{t \geq 0},  \mathbb{P}^\ep \right), \vre, \vue, W_\ep  \right].
\]
In addition, in view of the representation
theorem of Jakubowski \cite{Jakub} and the way the weak solutions are being constructed in \cite{BrHo}, we may assume, without lost of generality, that stochastic basis $\StoB$
as well as the Wiener process
$W$ coincide for all $\ep > 0$.

\subsection{Solutions of the Euler system}

Assume that we are given the stochastic basis $\StoB$ and the Wiener process $W$ identified in the preceding section.
Similarly to Definition \ref{def:strsol}, we introduce the (local) strong solutions of the Euler system (\ref{Eeq1}--\ref{Eeq2}):

\begin{Definition}\label{def:strsolE}

Let $\StoB$ be a stochastic basis with a complete right-continuous filtration, let $W$ be an $\left\{ \mathfrak{F}_t \right\}_{t \geq 0} $-cylindrical Wiener process. A
stochastic process $\vc{v}$ with a stopping time $\mathfrak{t}$ is called a (local) strong solution
to the Euler system (\ref{Eeq1}), (\ref{Eeq2}) provided
\begin{itemize}
\item the velocity $\vc{v} \in C([0,T]; W^{3,2}(\tor; \mathbb{R}^3))$ $\mathbb{P}$-a.s. is $\left\{ \mathfrak{F}_t \right\}_{t \geq 0}$-adapted,
\[
\E\bigg[ \sup_{t \in [0,T]} \| \vc{v} (t, \cdot) \|_{W^{3,2} (\tor; \mathbb{R}^3 )}^p \bigg] < \infty \quad \mbox{for all}\quad 1 \leq p < \infty;
\]
\item There holds $\p$-a.s.

\begin{equation} \label{Form}
\begin{split}
\Div \vc{v} &= 0, \\
\vc{v} (t \wedge \mathfrak{t})  &= \vc{v} (0) - \int_0^{t \wedge \mathfrak{t}} {\vc{P}_H} \left[ \vc{v} \cdot \Grad \vc{v} \right] \dt  +
\int_0^{t \wedge \mathfrak{t}} \vc{P}_H \left[ {\tn{G}}(1,\vc{v} ) \right] \, \Dif W,
\end{split}
\end{equation}
 a.e. in $(0,T)\times\tor$.
Here $\vc{P}_H$ denotes the standard Helmholtz projection onto the space of solenoidal functions.
\end{itemize}
\end{Definition}

The existence of local-in-time strong solutions to the stochastic Euler system was established by Glatt-Holtz and Vicol \cite[Theorem 4.3]{GHVic}
under certain restrictions imposed on the forcing coefficients $\tn{G}$. Here, we assume a very simple form of $\tn{G}$, namely that it is an affine function of the momentum
\begin{equation} \label{Diff}
\tn{G}(1, \vc{v})  = \tn{F} + \vc{v} \tn{H},\ \mbox{where} \ \tn{F} = \{ H_k \}_{k \geq 1},\ \tn{H} = \{ H_k \}_{k \geq 1},
\end{equation}
where $F_k$, $H_k$ are real numbers such that $\sum_{k\geq1}|F_k|<\infty$ and $\sum_{k\geq1}|H_k|<\infty$. The advantage of such a choice is that the pressure $\Pi$ can be computed explicitly from \ref{Form}.
Indeed seeing that
\[
\vc{P}_H \left[ {\tn{G}}(1,\vc{v} ) \right] =  {\tn{G}}(1,\vc{v} ),
\]
we get
\begin{equation} \label{presE}
\Grad \Pi = - \vc{P}^\perp_H [\vc{v} \cdot \Grad \vc{v}] = - \Grad \Delta^{-1} \Div (\vc{v} \otimes \vc{v}).
\end{equation}
Accordingly, the second equation in (\ref{Form}) reads
\begin{equation} \label{momEu}
\vc{v} (t \wedge \mathfrak{t})  = \vc{v} (0) - \int_0^{t \wedge \mathfrak{t}}  \left[ \vc{v} \cdot \Grad \vc{v} \right] \dt -
 \int_0^{t \wedge \mathfrak{t}} \Grad \Pi \ \dt +
\int_0^{t \wedge \mathfrak{t}}  {\tn{G}}(1,\vc{v} )  \,\Dif W.
\end{equation}

\subsection{Relative energy inequality}

Now, we are ready to apply the relative entropy inequality. Suppose that
$\vc{v}$, with a stopping time $\mathfrak{t}$ is a local strong solution of the Euler system (\ref{Eeq1}), (\ref{Eeq2}).
For each $M > 0$ let
\[
\tau_M = \inf_{t \in [0,T]} \left\{ \| \Grad \vc{v}(t, \cdot) \|_{L^\infty(\tor, R^3)} > M \right\}
\]
be another stopping time. In view of the existence result \cite[Theorem 4.3]{GHVic} we may assume, without loss of generality, that
\[
\tau_M \leq \mathfrak{j}.
\]

With the ansatz of test functions $r = 1$, $\vc{U}(t) = \vc{v}(t \wedge \tau_M) $,
\[
\mathcal{E} \left( \vr, \vu \ \Big| \ 1 , \vc{v} \right) \equiv \intTor{ \left[ \frac{1}{2} \varrho | \vu - {\bf v} |^2 + \frac{1}{\ep^2} \left( H(\varrho) -
H'(1) (\varrho - 1) - H(1) \right)  \right] }
\]
the relative energy inequality reads
\begin{equation} \label{REN1}
\begin{split}
\mathcal{E} \left( \vr, \vu \ \Big| \ 1 , \vc{v} \right) (\tau \wedge \tau_M)
&+ \int_0^{\tau \wedge \tau_M}  \intTor{ \Big( \tn{S}_\ep (\Grad \vc{v}) - \tn{S}_\ep (\Grad \vu) \Big): \Big( \Grad \vc{v} - \Grad \vu \Big) } \ \dt  \\
&\leq \mathcal{E} \left( \vr, \vu \ \Big| \ 1 , \vc{v} \right)(0) + M_R(\tau \wedge \tau_M) - M_R(0)\\
& - \int_0^{\tau \wedge \tau_M}  \intTor{ \vr (\vu - \vc{v}) \cdot \Grad \vc{v} \cdot (\vu -  \vc{v}) } \ \dt\\
&+ \int_0^{\tau \wedge \tau_M}  \intTor{ \tn{S}_\ep(\Grad \vc{v}):(\Grad \vc{v}-\Grad \vu)}\ \dt\\& - \int_0^{\tau \wedge \tau_M} \intTor{ \varrho\Grad \Pi \cdot (\vc{v} -\vu) }\ \dt \\
&+ \frac{1}{2}\sum_{k\geq1}\int_0^{\tau \wedge \tau_M} \intTor{ \varrho \Big|\frac{1}{{\varrho}}  {\vc{G}_k}(\varrho,\varrho \vu) - \vc{G}_k(1, \vc{v}) \Big|^2 }\ \dt.
\end{split}
\end{equation}

We show that, similarly to the proof of Theorem \ref{thm:uniq}, the terms of the right-hand side of (\ref{REN1}) can be ``absorbed'' by means of
a Gronwall type argument. To see this, we first observe that
\begin{equation} \label{Step1}
\begin{split}
\left| \int_0^{\tau \wedge \tau_M}  \intTor{ \vr (\vu - \vc{v}) \cdot \Grad \vc{v} \cdot (\vu -  \vc{v}) } \ \dt \right|
&\leq c \sup_{t \in [0,\tau_M]} \| \Grad \vc{v} \|_{L^\infty(\tor, R^3)} \int_0^{\tau \wedge \tau_M}
\mathcal{E} \left( \vr, \vu \ \Big| \ 1 , \vc{v} \right) \ \dt \\
& \leq c M \int_0^{\tau \wedge \tau_M}
\mathcal{E} \left( \vr, \vu \ \Big| \ 1 , \vc{v} \right) \ \dt.
\end{split}
\end{equation}

Similarly,
\begin{align}
&\left| \int_0^{\tau \wedge \tau_M}  \intTor{ \tn{S}_\ep(\Grad \vc{v}):(\Grad \vc{v}-\Grad \vu)} \ \dt \right| \nonumber\\ &\leq
\frac{1}{2} \int_0^{\tau \wedge \tau_M} \intTor{ \Big( \tn{S}_\ep (\Grad \vc{v}) - \tn{S}_\ep (\Grad \vu) \Big): \Big( \Grad \vc{v} - \Grad \vu \Big) } \ \dt + c \int_0^{\tau \wedge \tau_M} \intTor{ \left| \mathbb{S}(\Grad \vc{v}) \right|^2 } \ \dt\nonumber\\ &\leq
\frac{1}{2} \int_0^{\tau \wedge \tau_M} \intTor{ \Big( \tn{S}_\ep (\Grad \vc{v}) - \tn{S}_\ep (\Grad \vu) \Big): \Big( \Grad \vc{v} - \Grad \vu \Big) } \ \dt + (\mu_\ep + \eta_\ep) c T M^2;\label{Step2}
\end{align}
whence (\ref{REN1}) reduces to
\begin{equation} \label{REN2}
\begin{split}
\mathcal{E} \left( \vr, \vu \ \Big| \ 1 , \vc{v} \right) (\tau \wedge \tau_M)
&+ \frac{1}{2} \int_0^{\tau \wedge \tau_M}  \intTor{ \Big( \tn{S}_\ep (\Grad \vc{v}) - \tn{S}_\ep (\Grad \vu) \Big): \Big( \Grad \vc{v} - \Grad \vu \Big) } \ \dt  \\
&\leq \mathcal{E} \left( \vr, \vu \ \Big| \ 1 , \vc{v} \right)(0) + M_R(\tau \wedge \tau_M) - M_R(0)\\
& + c M \int_0^{\tau \wedge \tau_M}
\mathcal{E} \left( \vr, \vu \ \Big| \ 1 , \vc{v} \right) \ \dt + (\mu_\ep + \eta_\ep) c T M^2 \\
& - \int_0^{\tau \wedge \tau_M} \intTor{ \varrho\Grad \Pi \cdot (\vc{v} -\vu) }\ \dt \\
&+ \frac{1}{2}\sum_{k\geq1}\int_0^{\tau \wedge \tau_M} \intTor{ \varrho \Big|\frac{1}{{\varrho}}  {\vc{G}_k}(\varrho,\varrho \vu) - \vc{G}_k(1, \vc{v}) \Big|^2 }\ \dt
\end{split}
\end{equation}

Next, the integral containing the pressure can be written as
\[
\begin{split}
\int_0^{\tau \wedge \tau_M} \intTor{ \varrho\Grad \Pi \cdot (\vc{v} -\vu) }\ \dt &=
\int_0^{\tau \wedge \tau_M} \intTor{ \varrho\Grad \Pi \cdot \vc{v} } \ \dt \\
& - \int_0^{\tau \wedge \tau_M} \intTor{ \varrho\Grad \Pi \cdot \vu }\ \dt \\
& = \ep \int_0^{\tau \wedge \tau_M} \intTor{ \frac{ \varrho - 1}{\ep} \Grad \Pi \cdot \vc{v} } \ \dt
- \int_0^{\tau \wedge \tau_M} \intTor{ \varrho\Grad \Pi \cdot \vu }\ \dt.
\end{split}
\]

Finally, we handle the integral
\[
\sum_{k\geq1}\intTor{ \varrho \Big|\frac{1}{{\varrho}}  {\vc{G}_k}(\varrho,\varrho \vu) - \vc{G}_k(1, \vc{v}) \Big|^2 }.
\]
Motivated by the specific form of $\tn{G}(1, \vc{v})$ introduced in (\ref{Diff}), we restrict ourselves to
\[
\tn{G} (\vr, \vr \vu) = \vr \tn{F} + \vr \vu \tn{H};
\]
whence
\begin{align*}
\sum_{k\geq1}& \intTor{ \varrho \Big|\frac{1}{{\varrho}}  {\vc{G}_k}(\varrho,\varrho \vu) - \vc{G}_k(1, \vc{v}) \Big|^2 }  \\&=\sum_{k\geq1}
\intTor{ \varrho \left| (\vu - \vc{v}) H_k \right|^2 } \leq \,c\, \mathcal{E} \left( \vr, \vu \ \Big| \ 1 , \vc{v} \right)
\end{align*}
using $\sum_{k\geq1}|H_k|^2$.
Consequently, the relation (\ref{REN2}) gives rise to
\begin{align} \label{REN3}
\expe{\mathcal{E} \left( \vr, \vu \ \Big| \ 1 , \vc{v} \right) (\tau \wedge \tau_M)} &\leq c(M,T) \left(
\expe{\mathcal{E} \left(  \vr, \vu \ \Big| \ 1 , \vc{v} \right) (0)}  + \mu_\ep + \eta_\ep \right)\\
&+\ep \expe{ \int_0^{\tau \wedge \tau_M} \intTor{ \frac{ \varrho - 1}{\ep} \Grad \Pi \cdot \vc{v} } \ \dt }
- \expe{ \int_0^{\tau \wedge \tau_M} \intTor{ \varrho\Grad \Pi \cdot \vu }\ \dt }.\nonumber
\end{align}
In order to control the last two terms in (\ref{REN3}), we evoke again (\ref{REI2}), this time for $r = 1$, $\vc{U} = 0$ obtaining
\[
\begin{split}
&\expe{ \intTor{ \left[ \frac{1}{2} \vr |\vu|^2 + \frac{1}{\ep^2} \left( H(\vr) - H'(1)(\vr - 1) - H(1) \right) \right] }(\tau \wedge \tau_M) }\\
&\leq
\expe{ \intTor{ \left[ \frac{1}{2} \vr |\vu|^2 +  \frac{1}{\ep^2} \left( H(\vr) - H'(1)(\vr - 1) - H(1) \right) \right] }(0) }.
\end{split}
\]
Thus, if the right-hand side of the above inequality is bounded uniformly for $\ep \to 0$, we deduce from (\ref{presE}) that
\[
\left| \expe{ \int_0^{\tau \wedge \tau_M} \intTor{ \frac{ \varrho - 1}{\ep} \Grad \Pi \cdot \vc{v} } \ \dt } \right| \leq c
\]
uniformly for $\ep \to 0$, and
\[
\vr_\ep \vu_\ep \to \vc{v} \ \mbox{weakly in}\ L^{\frac{2 \gamma}{\gamma +1} }((0,T) \times \tor \times \Omega).
\]
In particular, the last two terms on the right-hand side of (\ref{REN3}) vanish for $\ep \to 0$.

We have proved the following result.

\begin{Theorem} \label{thm:sing}
Let $\tn{G}$ be given as
\[
\tn{G}(\vr, \vr \vu) = \vr \tn{F} + \vr \vu \tn{H},\quad \sum_{k \geq 1} \left( |F_k| + |H_k| \right) < \infty.
\]
Let $\StoB$ be a stochastic basis with a complete right-continuous filtration. Let the initial data $\vr_{0,\ep}$, $(\vr \vu)_{0,\ep}$ be given such that
\[
\vr_{0, \ep}, (\vr \vu)_{0,\ep}  \in L^\gamma(\tor) \times L^{\frac{2 \gamma}{\gamma + 1}}(\tor;
\mathbb{R}^3) \ \Big| \ \vr_{0,\ep} \geq \underline \vr > 0,\ \frac{|\vr_{0,\ep} - 1|}{\ep} \leq \delta(\ep),\
|(\vr \vu)_{0,\ep} - \vc{v}_0 |
\leq \delta (\ep) \ \mathbb{P}-\mbox{a.s.}
\]
where
\[
\delta (\ep) \to 0 \ \mbox{as}\ \ep \to 0,
\]
and where
$\vc{v}_0$ is an $\mathfrak{F}_0$-measurable random variable,
\begin{align*}
&\vc{v}_0 \in W^{3,2}(\tor;R^3), \ \Div \vc{v}_0 = 0 \quad \text{$\mathbb{P}$-a.s.},\\
&\expe{ \| \vc{v}_0 \|_{W^{3,2}(\tor; R^3)}^p } < \infty \quad \mbox{for all}\quad 1 \leq p < \infty.
\end{align*}

Then the scaled Navier-Stokes system (\ref{seq1}--\ref{seq3}) with
\[
\mu_\ep > 0,\ \eta_\ep \geq 0,\ \mu_\ep \to 0,\ \eta_\ep \to 0 \ \mbox{as}\ \ep \to 0,
\]
admits a family of (weak) dissipative martingale solutions
\[
\left[ (\tilde {\Omega}, \tilde {\mathfrak{F} }, \left\{ \tilde{ \mathfrak{F}}_t \right\}_{t \geq 0},
\tilde{ \mathbb{P}}, \vre, \vre \vue, W
\right]_{\ep > 0}
\]
defined $(0,T) \times \tor$ and with the initial law
\[
\Lambda_\ep = \mathbb{P} \left[ \vr_{0,\ep}, (\vr \vu)_{0,\ep} \right]^{-1},
\]
such that
\begin{equation} \label{final}
\sup_{t \in [0,T]} \expe{ \intTor{ \left[ \frac{1}{2} \vre | \vue - \vc{v} |^2 + \frac{1}{\ep} \left( H(\vre) - H'(1) (\vre - 1) - H(1) \right) \right] } (t \wedge \mathfrak{t}) }
\to 0
\end{equation}
as $\ep \to 0$, where
$\vc{v}$, with a positive stopping time $\mathfrak{t}$, is a local regular solution of the
Euler system (\ref{Eeq1}), (\ref{Eeq2}), with the initial velocity $\vc{v}(0, \cdot)$ satisfying
\[
\tilde{\mathbb{P} }\left[ \vc{v}(0, \cdot) \right]^{-1} =
\mathbb{P} \left[ \vc{v}_0 \right]^{-1}.
\]

\end{Theorem}

\begin{Remark}

It follows from (\ref{final}) that
\[
\expe{ \int_0^{T \wedge \mathfrak{t}} \| \vue - \vc{v} \|^2_{L^2(\tor; \mathbb{R}^3)} \ \dt } \to 0 \ \mbox{as} \ \ep \to 0.
\]

\end{Remark}

\begin{Remark}

The situation considered in Theorem \ref{thm:sing} corresponds to the so-called well-prepared data. The ill-prepared data
generating fast frequency acoustic waves will be treated elsewhere.

\end{Remark}

\begin{Remark}
\begin{itemize}
\item Note that the inviscid limit in the purely \emph{incompressible} setting was studied by Glatt-Holtz, \v Sver\' ak, and Vicol \cite{GHSvVi} in the two-dimensional setting.
\item We studied the incompressible limit of the compressible Navier--Stokes
with stochastic forcing in our previous paper \cite{BrFH}.
\end{itemize}
\end{Remark}

\centerline{\bf Acknowledgement}
\noindent{D.B. was partially supported by Edinburgh Mathematical Society.}\\
{The research of E.F. leading to these results has received funding from the European Research Council under the European Union's Seventh Framework Programme (FP7/2007-2013)/ ERC Grant Agreement 320078. The Institute of Mathematics of the Academy of Sciences of the Czech
Republic is supported by RVO:67985840.}\\\

\def\cprime{$'$} \def\ocirc#1{\ifmmode\setbox0=\hbox{$#1$}\dimen0=\ht0
  \advance\dimen0 by1pt\rlap{\hbox to\wd0{\hss\raise\dimen0
  \hbox{\hskip.2em$\scriptscriptstyle\circ$}\hss}}#1\else {\accent"17 #1}\fi}

\end{document}